\documentclass[reqno,11pt]{amsart}
\usepackage{amssymb}
\usepackage{mathrsfs}
\usepackage{}
\usepackage{amsmath,amssymb}
\usepackage{paralist}
\usepackage{graphics} 
\usepackage{epsfig} 
\usepackage[colorlinks=true]{hyperref}
\hypersetup{urlcolor=red, citecolor=blue}
\usepackage[top=1in,bottom=1in,left=1in,right=1in]{geometry}
\usepackage{cite}
\newtheorem{thm}{Theorem}[section]

\newtheorem{lem}[thm]{Lemma}
\newtheorem{prop}[thm]{Proposition}
\theoremstyle{definition}
\newtheorem{defi}{Definition}[section]
\newtheorem{re}{Remark}[section]
\numberwithin{equation}{section}
 \newcommand{\be}{\begin{equation}}
 \newcommand{\ee}{\end{equation}}
 \newcommand{\bes}{\begin{eqnarray}}
 \newcommand{\ees}{\end{eqnarray}}
 \newcommand{\bess}{\begin{eqnarray*}}
 \newcommand{\eess}{\end{eqnarray*}}
 
\begin{document}
\title[3D density-dependent incompressible Boussinesq system]
      {The decay and stability of solutions for the 3D density-dependent incompressible Boussinesq system}%
\author[Liu]{Xiaopan Liu}%
\address{Department of Mathematics, Henan Normal University, Xinxiang 453000,   PR China.}
\email{liuxiaopan112@126.com}

\author[Zhang]{Qingshan Zhang }%
\address{Department of Mathematics, Henan Institute of Science and Technology, Xinxiang 453003, PR China.}
\email{qingshan11@yeah.net}

\thanks{This work is supported by China Postdoctoral Science Foundation (No. 2018M630824) and key teacher in Colleges and Universities training plan Henan province (No. 2020GGJS164).}



\begin{abstract}
This paper deals with stability and the large-time decay to any given global smooth solutions of the 3D density-dependent incompressible
Boussinesq system. The decay rate for solutions of the corresponding Cauchy problem is obtained in this work. With the aid of this decay rate,
it is shown that a small perturbation of initial data $(\overline{a}_0,\overline{\theta}_0, \overline{u}_0)$ still generates a global smooth
solution to the density-dependent Boussinesq system, and this solution keeps close to the reference solution.
\end{abstract}
\maketitle

\section{Introduction }
\setcounter{equation}{0}
\setlength{\baselineskip}{17pt}{\setlength\arraycolsep{2pt}
The Boussinesq system describes the movement of incompressible fluid under the influence of gravitational forces, we refer to \cite{P} for more applications in fluid mechanics. In the three-dimensional case, the inhomogeneous incompressible Boussinesq system is of the form
\begin{eqnarray}\label{sys:Boussinesq}
\left\{\begin{array}{lll}
 \medskip
\partial_t \rho + u\cdot\nabla \rho=0, \qquad (t,x)\in \mathbb{R^+}\times\mathbb{R}^3,\\
 \medskip
\rho\partial_t \theta+ \rho(u\cdot\nabla \theta)-\Delta\theta=0,\\
 \medskip
\rho\partial_t u+ \rho(u\cdot\nabla u)-\Delta u+ \nabla\Pi=\rho \theta e_3,\\
 \medskip
\mathrm{div} u=0,
\end{array}\right.
\end{eqnarray}
where the unknown $\rho=\rho(x,t), u=u(x,t)$ denote the density and velocity vector-field separately, and $\theta=\theta(x,t)$ is a
scalar quantity such as the concentration of a chemical substance or the temperature variation in a gravity field, in which case
$\rho\theta e_3$ with $e_3=(0,0,1)$ represents the buoyancy force and the gradient of the pressure $\nabla\Pi$ is the Lagrangian multiplier associated to the divergence free constraint over the velocity.

The main result of this paper concerns with the large-time decay and stability of large global smooth solutions for the inhomogeneous Boussinesq system (\ref{sys:Boussinesq}) under suitable small perturbations. As far as I know, there are only some result about homogeneous Boussinesq system, for example\cite{L-G,L-L} and the references therein. As a subsystem of (\ref{sys:Boussinesq}) the homogeneous Navier-Stokes equations have been investigated in the past few years, global stability of solutions can be founded in \cite{GIP,GZ,PRS}. For the inhomogeneous Navier-Stokes equations, when the initial density is close enough to a positive constant, Danchin in \cite{Dan3} proved that the system has a unique local-in-time solution. This result is improved by Abidi in \cite{A}, where he prove that the inhomogeneous Navier-Stokes system has a unique global solution under appropriate assumptions on the initial data. When the initial density is close enough or not to a positive constant, more about the global and uniqueness solutions for the inhomogeneous Navier-Stokes system can be found in \cite{AGZ1,AGZ2,AP,B,CPZ,Pa-Z,ZY}. The large-time decay and stability to any given global smooth solutions of the 3D incompressible inhomogeneous Navier-Stokes equations can be found in \cite{AGZ,AZ}.

The main motivation of studying the stability of global strong solutions for Boussinesq equations is that: When the global
existence problem of smooth solutions is not completely solved, but one has global existence of solutions presenting certain type of
symmetry \cite{Leonardi-Malek-Necas-Pokorny-1999,Shirota-Yanagisawa-1994,WWL}, this kind of stability becomes particularly interesting because it
provides a non-symmetric large strong global solution as a small perturbation of a symmetric one. This is the case of the Navier-Stokes equations that have global large strong solutions with axial, rotational, and helical symmetry, see for example \cite{AGZ,GZ,PRS}. Recently, there are some results concerning global well-posedness for the Boussinesq system (\ref{sys:Boussinesq}) with large axisymmetric data in three dimension, see \cite{FLZ,Hmidi-Rousset-2010,Hmidi-Rousset-2011}. It is natural to consider the stability of global strong solutions for the Boussinesq system (\ref{sys:Boussinesq}).

When the solution is small and $\rho=1$, there are some stability results. Ferreira and Villamizar-Roa give a class of stable steady solutions in
weak-$L^p$ spaces, in the sense that they only assume that the stable steady solution belongs to scaling invariant class $L^{(n,\infty)}_{\sigma}\times L^{(n,\infty)}$ \cite{FV1}. The authors in \cite{FF} investigated well-posedness of mild solution and existence of self-similar ones in the framework of Morrey space. We point out that these small global solutions in weak-$L^p$ and Morrey spaces may be large (even unbounded) in the classical norms $L^2,\ H^1$ and in Besov spaces with positive regularity. More results about the stability of small global solutions for the Boussinesq system can be found in \cite{H2,H3,M,QZ} and the references therein.

The goal of the present paper is to give the decay and stability of solutions for system (\ref{sys:Boussinesq}). Let $a=\frac{1}{\rho}-1$, system (\ref{sys:Boussinesq}) can be equivalently reformulated as
\begin{eqnarray}\label{sys:a-Boussinesq}
\left\{\begin{array}{lll}
 \medskip
\partial_t a + u\cdot\nabla a=0, \qquad (t,x)\in \mathbb{R}^+\times\mathbb{R}^3,\\
 \medskip
\partial_t \theta+ (u\cdot\nabla \theta)-(1+a)\Delta\theta=0,\\
 \medskip
\partial_t u+ (u\cdot\nabla u)-(1+a)(\Delta u-\nabla\Pi)=\theta e_3,\\
 \medskip
\mathrm{div} u=0.
\end{array}\right.
\end{eqnarray}
Assume $(a,\theta,u)$ is a smooth global solution of the above system, we have
$$
a\Delta \theta=\nabla(a\nabla\theta)-\nabla a\nabla\theta,
$$
then the inhomogeneous incompressible Boussinesq system (\ref{sys:a-Boussinesq}) can be rewritten as follows
\begin{eqnarray}\label{rewrite-boussinesq}
\left\{\begin{array}{lll}
 \medskip
\partial_t a + u\cdot\nabla a=0, \qquad (t,x)\in \mathbb{R^+}\times\mathbb{R}^3,\\
 \medskip
\partial_t \theta+ (u\cdot\nabla \theta)-\Delta\theta-\kappa\nabla(a\nabla\theta)+\nu\nabla a\nabla\theta=0,\\
 \medskip
\partial_t u+ (u\cdot\nabla u)-(1+a)\Delta u+ \nabla\Pi=\theta e_3,\\
 \medskip
\mathrm{div} u=0,
 \end{array}\right.
\end{eqnarray}
where $\kappa=\nu=1$. In what follows, we mainly prove the decay rate and stability of global solutions for (\ref{rewrite-boussinesq}) with $\kappa=1$,  $\nu=0$. Our main results describe the decay of the solution as follows. We assume that the solutions are global existence. We obtain the decay and  boundedness of the global solutions in the following theorem.

\begin{thm}\label{main-decay-bounded}
Let $m\leq a_0\leq M$, $(a_0,\theta_0, u_0)\in{B}_{2,1}^{\frac{3}{2}}\times L^2\times L^2_{\sigma}$, there exist absolute constants $\varepsilon_0>0$ and $M_1>0$ such that
$$
\|a_0\|_{\dot{B}_{2,1}^{\frac{3}{2}}}<\eta_0,
$$
and
$$
\int_0^\infty\|\nabla u\|_{\infty}d\tau<M_1\label{u-satisfy}.
$$
In addition, assume the initial data $(a_0,\theta_0,u_0)$ satisfy $(\theta_0,u_0)\in (H^1)^2$, $\theta_0\in L^1\cap L^1_1$ and $\int\theta_0dx=0$, there exists absolute constants $\varepsilon_0>0$ such that
\begin{eqnarray*}
\|\theta_0\|_1<\varepsilon_0.
\end{eqnarray*}
Then there hold:\\
{\rm{(a)}} The weak solution constructed in Proposition \ref{1.1} satisfies,
\begin{eqnarray*}
\|\theta(t)\|_2^2\leq C(1+t)^{-\frac{5}{2}},\quad \|\nabla\theta(t)\|_2^2\leq C(1+t)^{-\frac{7}{2}},
\end{eqnarray*}
and
\begin{eqnarray*}
 \|u(t)\|_2^2\leq C,\quad t>0
\end{eqnarray*}
for some $C>0$.
{\rm{(b)}} The above solution is uniformly bounded, and satisfy
\begin{eqnarray*}
&& a\in L^\infty([0,\infty); W^{1,3}(\mathbb{R}^3))\cap L^\infty([0,\infty); L^2(\mathbb{R}^3)),\\
&&\theta\in L^\infty([0,\infty); H^1(\mathbb{R}^3))\cap L_{loc}^2(\mathbb{R}^+{\dot{B}}^{2}_{2,2}(\mathbb{R}^3)),\\
&&u\in L^\infty([0,\infty);H^1(\mathbb{R}^3))\cap L_{loc}^2(\mathbb{R}^+; {\dot{B}}^{2}_{2,2}(\mathbb{R}^3)),\\
&&\nabla \Pi\in L_{loc}^2(\mathbb{R}^+;{\dot{B}}^{1}_{2,2}(\mathbb{R}^3)).
\end{eqnarray*}
{\rm{(c)}} Assume additionally that $\theta_0\in
\dot{B}^{-\frac{3}{2}}_{2,1}(\mathbb{R}^3)$ and there
exist a absolute constants  $M_2>0$ and a large time $T_0^*$ such that
$$
\int_0^{T_0^*}\|\nabla u\|_{\dot{B}_{2,1}^{\frac{3}{2}}}d\tau<M_2\label{u-satisfy-1},
$$
we have
\begin{eqnarray*}
&&a\in C_b([0,\infty); B^{\frac{3}{2}}_{2,1}(\mathbb{R}^3)),\\
&&\theta\in C_b([0,\infty); \dot{B}^{-\frac{3}{2}}_{2,1}(\mathbb{R}^3))\cap L_{loc}^1(\mathbb{R}^+;{\dot{B}}^{\frac{1}{2}}_{2,1}(\mathbb{R}^3)),\\
&&u\in C_b([0,\infty);{B}^{\frac{1}{2}}_{2,1}(\mathbb{R}^3))\cap L_{loc}^1(\mathbb{R}^+; {\dot{B}}^{\frac{5}{2}}_{2,1}(\mathbb{R}^3)),\\
&&\nabla\Pi\in L_{loc}^1(\mathbb{R}^+; {{B}}^{\frac{1}{2}}_{2,1}(\mathbb{R}^3)).
\end{eqnarray*}
{\rm{(d)}} Moreover, if $(a_0,\theta_0, u_0)\in B^{\frac{5}{2}}_{2,1}(\mathbb{R}^3)\times\dot{B}^{-\frac{1}{2}}_{2,1}(\mathbb{R}^3)\times
B^{\frac{3}{2}}_{2,1}(\mathbb{R}^3)$ besides the conditions in (c), we have
\begin{eqnarray*}
&& a\in C_b([0,\infty);B^{\frac{5}{2}}_{2,1}(\mathbb{R}^3)), \\
&&\theta\in C_b([0,\infty); \dot{B}^{-\frac{1}{2}}_{2,1}(\mathbb{R}^3))\cap L_{loc}^1(\mathbb{R}^+;{{B}}^{\frac{3}{2}}_{2,1}(\mathbb{R}^3)),\\
&&u\in C_b([0,\infty);{B}^{\frac{3}{2}}_{2,1}(\mathbb{R}^3))\cap L_{loc}^1(\mathbb{R}^+; {\dot{B}}^{\frac{7}{2}}_{2,1}(\mathbb{R}^3)),\\
&&\nabla\Pi\in L_{loc}^1(\mathbb{R}^+; {{B}}^{\frac{3}{2}}_{2,1}(\mathbb{R}^3)).
\end{eqnarray*}
\end{thm}
Next, we give our main results for the stability of the global solution constructed in Theorem \ref{main-decay-bounded}.
\begin{thm}\label{1.2}
{\rm{(a)}} Let $(\bar{a},\bar{\theta},\bar{u},\bar{\Pi})$ be a global solution of (\ref{rewrite-boussinesq}) constructed in Theorem \ref{main-decay-bounded} (c). If there exist two positive constants $\varepsilon_0$ and $\varepsilon_1$ such that $\|\bar{\theta}_0\|_1\leq \varepsilon_0$ and
\begin{eqnarray*}
\|\tilde{a}_0\|_{\dot{B}_{2,1}^{\frac{3}{2}}}
 +\|\tilde{\theta}_0\|_2+\|\tilde{\theta}_0\|_{\dot{B}_{2,1}^{-\frac{3}{2}}}
 +\|\tilde{u}_0\|_{{B}_{2,1}^{\frac{1}{2}}}<\varepsilon_1,
\end{eqnarray*}
then $(a_0,\theta_0,u_0)=(\bar{a}_0+\tilde{a}_0,\bar{\theta}_0+\tilde{\theta}_0, \bar{u}_0+\tilde{u}_0)$ generates a global solution $(a,\theta, u, \Pi)$ of (\ref{rewrite-boussinesq}) with $\kappa=1$, $\nu=0$, which satisfies the stability
estimate
\begin{eqnarray*}
&&\|a-\bar{a}\|_{\tilde{L}^\infty(\mathbb{R}^+;{B}^{\frac{3}{2}}_{2,1}(\mathbb{R}^3))}+
  \|u-\bar{u}\|_{\tilde{L}^\infty(\mathbb{R}^+;{B}^{\frac{1}{2}}_{2,1}(\mathbb{R}^3))}+
  \|u-\bar{u}\|_{L^1(\mathbb{R}^+;{\dot{B}}^{\frac{5}{2}}_{2,1}(\mathbb{R}^3))}\\
&&+\|\nabla\Pi-\nabla\bar{\Pi}\|_{L^1(\mathbb{R}^+;{\dot{B}}^{\frac{1}{2}}_{2,1}(\mathbb{R}^3))}
       +\|\theta-\bar{\theta}\|_{\tilde{L}^\infty(\mathbb{R}^+;\dot{B}^{-\frac{3}{2}}_{2,1}(\mathbb{R}^3))}
  +\|\theta-\bar{\theta}\|_{L^1(\mathbb{R}^+;{\dot{B}}^{\frac{1}{2}}_{2,1}(\mathbb{R}^3))}\\
&&+\|\theta-\bar{\theta}\|_2\leq C\varepsilon_1,
\end{eqnarray*}
for some positive constant C.

{\rm{(b)}} Let $(\bar{a},\bar{\theta},\bar{u},\bar{\Pi} )$ be a global solution of constructed in Theorem \ref{main-decay-bounded} (d).
If there exist two positive constants $\varepsilon_0^*$ and $\varepsilon_1^*$ satisfying
$\|\bar{\theta}_0\|_1\leq \varepsilon_0^*$ and
\begin{eqnarray*}
   \|\tilde{a}_0\|_{\dot{B}_{2,1}^{\frac{5}{2}}}
   +\|\tilde{\theta}_0\|_{\dot{B}_{2,1}^{-\frac{1}{2}}}
   +\|\tilde{u}_0\|_{{B}_{2,1}^{\frac{3}{2}}}<\varepsilon_1^*,
\end{eqnarray*}
then $(a_0,\theta_0,u_0)=(\bar{a}_0+\tilde{a}_0,\bar{\theta}_0+\tilde{\theta}_0,
\bar{u}_0+\tilde{u}_0)$ generates a global solution of (\ref{rewrite-boussinesq}) with $\kappa=1$,  $\nu=0$, which satisfies the stability estimate
\begin{eqnarray}\label{more}
&&\|a-\bar{a}\|_{\tilde{L}^\infty(\mathbb{R}^+;{B}^{\frac{5}{2}}_{2,1}(\mathbb{R}^3))}+\|u-\bar{u}\|_{\tilde{L}^\infty(\mathbb{R}^+;{B}^{\frac{3}{2}}_{2,1}(\mathbb{R}^3))}
  +\|u-\bar{u}\|_{L^1(\mathbb{R}^+;{\dot{B}}^{\frac{7}{2}}_{2,1}(\mathbb{R}^3))}
  \nonumber\\
&&+\|\nabla\Pi-\nabla\bar{\Pi}\|_{L^1(\mathbb{R}^+;{\dot{B}}^{\frac{3}{2}}_{2,1}(\mathbb{R}^3))}+\|\theta-\bar{\theta}\|_{\tilde{L}^\infty(\mathbb{R}^+;\dot{B}^{-\frac{1}{2}}_{2,1}(\mathbb{R}^3))}
  +\|\theta-\bar{\theta}\|_{L^1(\mathbb{R}^+;{\dot{B}}^{\frac{3}{2}}_{2,1}(\mathbb{R}^3))}
  \leq C\varepsilon_1^*,
\end{eqnarray}
for some positive constant $C$.
\end{thm}

The paper is organized as follows. In Section 2, we recall some preliminaries of Besov spaces. In Section 3, we prove the existence and decay of weak solutions which is the main technique to prove stability for this system in critical Besov space. In Section 4, we give the proof of Theorem \ref{main-decay-bounded}. Finally,  stability of global solutions are proved in the Besov space in Section 5.

\section{Some preliminaries of Besov spaces}
\setcounter{equation}{0}
We first give some notations. Denote by $C_0^\infty$ the space of smooth
functions in $\mathbb{R}^3$ with compact support. The $L^p$ will be denoted by
$\|\cdot\|_p$. Let $\mathcal{V}=\{\phi\in C_0^\infty|\nabla\cdot\phi=0\}$.
Denote by $L^p_{\sigma}$ the completion of $\mathcal{V}$ under the norm
$\|\cdot\|_p$. We denote $$\|f\|_{L_r^p}=(\int |f(x)|^p(1+|x|^{pr}dx)^{1/p}.$$

For $a\preceq b$, we mean that there is a uniform constant $C$,
which may be different on different lines, such that $a \leq C b$.
We shall denote by $(a|b)$ (or $(a|b)_{L^2}(\mathbb{R}^3)$) the
$L^2(\mathbb{R}^3)$ inner product of $a$ and $b$. For $X$ a Banach
space and $I$ an interval of $\mathbb{R}$, we denote by $C(I;X)$ the
set of continuous functions on $I$ with values in $X$, and by
$C_b(I;X)$  the subset of bounded functions of $C(I;X)$. For $q\in
[1,\infty]$, the notation $L^q(I;X)$ stands for the set of
measurable functions on $I$ with values in $X$ such that $t\mapsto
\|f(t)\|_X$ belongs to $L^q(I)$. For $k
\in N$, we denote by $W^{k,p}$ the set of $L^p$ functions with
derivatives up to order $k$ in $L^p$.

Next, we recall some tools from the theories of the Besov spaces, for details see \cite{AGZ,BCD}.
\begin{defi}\label{2.1}
Let $u\in \mathcal{S}'(\mathbb{R}^N)$ such that $\lim_{j\rightarrow -\infty}\dot{S}_ju=0$, $s\mathbb{R}$ and $1 \leq
p,\ r \leq \infty$. We set
\begin{eqnarray*}
   \|u\|_{\dot{B}_{p,r}^s}=(\sum_{q\in \mathbb{Z}} 2^{qsr}\|\dot{\triangle}_q u\|_{L^p}^r)^\frac{1}{r} \ \ \ \mbox{if}\ \ \ r<\infty\ \ \ \mbox{and}
  \ \ \ \|u\|_{\dot{B}_{p,\infty}^s}=\sup_{q\in \mathbb{Z}} 2^{qs}\|\dot{\triangle}_q u\|_{L^p}.
\end{eqnarray*}
For $s<\frac{N}{p}\ (\mbox{or} \ s=\frac{N}{p} \ \mbox{if}\ r=1)$,
we define
\begin{eqnarray*}
   \dot{B}_{p,r}^s=\{u\in \mathcal{S}';\ \|u\|_{\dot{B}_{p,r}^s} <\infty\}.
\end{eqnarray*}
If $k\in \mathbb{N}$ and $\frac{N}{p}+k \leq s < \frac{N}{p}+k+1$
(or $s=\frac{N}{p}+k+1 $  if $r=1$), then $\dot{B}_{p,r}^s$ is
defined as the subset of distributions $u\in \mathcal{S}'(\mathbb{R}^N)$ such
that $\partial^\beta u\in \dot{B}_{p,r}^{s-k}$ whenever $|\beta|=k$.
\end{defi}

\begin{re} \label{2.2}
{\rm{(1)}} We point out that if $s > 0$, then $B^s_{p,r}= \dot{B}^s_{p,r}\cap
L^p$ and
\begin{eqnarray*}
   \|u\|_{B^s_{p,r}}\approx \|u\|_{\dot{B}^s_{p,r}}+\|u\|_{L^p}.
\end{eqnarray*}
If $s < 0$, then $ \dot{B}^s_{p,r}\hookrightarrow B^s_{p,r}$ and
\begin{eqnarray*}
   \|u\|_{B^s_{p,r}}\leq \frac{C}{-s}\|u\|_{\dot{B}^s_{p,r}}
\end{eqnarray*}
with $B^s_{p,r}$ being the nonhomogeneous Besov space.

{\rm{(2)}}(i) If $r_1\leq r_2$, for all $s\in \mathbb{R}$, we have
\begin{eqnarray*}
  \dot{B}_{p,r_1}^s\hookrightarrow \dot{B}_{p,r_2}^s.
\end{eqnarray*}

(ii) If $1\leq p_1\leq p_2\leq \infty,\ 1\leq r_1\leq r_2\leq
\infty$, for all $s\in \mathbb{R}$,
\begin{eqnarray*}
  \dot{B}_{p_1,r_1}^s\hookrightarrow \dot{B}_{p_2,r_2}^{s-N(\frac{1}{p_1}-\frac{1}{p_2})}.
\end{eqnarray*}

(iii) If $u\in \dot{B}_{p,r}^{s_1} \bigcap \dot{B}_{p,r}^{s_2}$,
then $u\in \dot{B}_{p,r}^{s}$ with $s=\alpha s_1+(1-\alpha)s_2$ for all $\alpha\in(0,1)$,
\begin{eqnarray*}
  \|u\|_{\dot{B}_{p,r}^{s}}\preceq \|u\|_{\dot{B}_{p,r}^{s_1}}^\alpha\|u\|_{\dot{B}_{p,r}^{s_2}}^{1-\alpha}.
\end{eqnarray*}

{\rm{(3)}} Let $s\in \mathbb{R}$, $1 \leq p,r\leq \infty$, and $u\in
S'(\mathbb{R}^3)$. Then $u$ belongs to
$\dot{B}^s_{p,r}(\mathbb{R}^3)$ if and only if there exists
$\{c_{j,r}\}_{j\in \mathbb{Z}} $ such that $\|c_{j,r}\|_{l^r}=1$ and we have the inequality
\begin{eqnarray*}
  \|\dot{\triangle}_j u\|_{L^p}\preceq c_{j,r}2^{-js}\|u\|_{\dot{B}^s_{p,r}}\ \ \ for\ all\ j\in \mathbb{Z}.
\end{eqnarray*}
\end{re}

We shall use frequently the product estimate in homogeneous space, here we give the definition of product,
$$uv=\sum_{j',j}\dot{\triangle}_{j'}u\dot{\triangle}_{j}v=\dot{T}_u v+\dot{T}_v u+\dot{R}(u,v),$$ where
$$\dot{T}_u v=\sum_j \dot{S}_{j-1} u\dot{\triangle}_{j}v, \ \dot{R}(u,v)=\sum_{|k-j|\leq 1}\dot{\triangle}_{k}u\dot{\triangle}_{j}v$$.

\begin{lem}\label{T-estimate}(\cite{BCD}{Theorem 2.47})
There exists a constant C such that for any real number $s$ and any $(p,r)$  in $[1,\infty]^2$, we have, for any $(u,v)\in L^\infty \times\dot{B}_{p,r}^{s},$
$$ \|\dot{T}_u v\|_{\dot{B}_{p,r}^{s}}\leq C^{1+|s|}\|u\|_{L^\infty}\|v\|_{\dot{B}_{p,r}^{s}}.$$
Moreover, for any $(s,t)$ in $\mathbb{R}\times (-\infty,0)$ and any $(p,r_1,r_2)$ in $[1,\infty]^3$, we have, for any $(u,v)\in \dot{B}_{\infty,r_1}^{t}\times \dot{B}_{p,r_2}^{s}$,
\begin{eqnarray*}
\|\dot{T}_u v\|_{\dot{B}_{p,r}^{s+t}}\leq C^{1+|s+t|}\|u\|_{\dot{B}_{\infty,r_1}^{t}}\|v\|_{\dot{B}_{p,r_2}^{s}}, \frac{1}{r}=\min\{1,\frac{1}{r_1}+\frac{1}{r_2}\}.
\end{eqnarray*}
\end{lem}

\begin{lem}\label{R-estimate}(\cite{BCD}{Theorem 2.52})
Let $(s_1,s_2)$ be in $\mathbb{R}^2$ and $(p_1,p_2,r_1,r_2)$ be in $[1,\infty]^4$. Assume that
$$ \frac{1}{p}=\frac{1}{p_1}+\frac{1}{p_2}\leq 1$$ and $$ \frac{1}{r}=\frac{1}{r_1}+\frac{1}{r_2}\leq 1.$$
If $s_1+s_2$ is positive, then we have, for any $(u,v)$
in $\dot{B}_{p_1,r_1}^{s_1}\times \dot{B}_{p_2,r_2}^{s_2}$,
$$
\|\dot{R}(u,v)\|_{\dot{B}_{p,r}^{s_1+s_2}}\leq\frac{C^{|s_1+s_2|+1}}{s_1+s_2}\|u\|_{\dot{B}_{p_1,r_1}^{s_1}}\|v\|_{\dot{B}_{p_2,r_2}^{s_2}}.
$$
When $r=1$ and $s_1+s_2\geq 0$, we have, for any $(u,v)$ in $\dot{B}_{p_1,r_1}^{s_1}\times \dot{B}_{p_2,r_2}^{s_2}$,
$$
\|\dot{R}(u,v)\|_{\dot{B}_{p,1}^{s_1+s_2}}\leq C^{|s_1+s_2|+1}\|u\|_{\dot{B}_{p_1,r_1}^{s_1}}\|v\|_{\dot{B}_{p_2,r_2}^{s_2}}.
$$
\end{lem}
As $s_1+s_2>0$ is a special case of Lemma \ref{T-estimate} and Lemma \ref{R-estimate}, the following lemma is given without proof.

\begin{lem}\label{3.1}
For all $s_1,\ s_2\in \mathbb{R},\ 1\leq p,\ p_1,\ p_2,\ r_1,\ r_2\leq
\infty$ satisfying
\begin{eqnarray*}
  s_1+s_2>0,\  s_j<\frac{N}{p_j}\   or \ s_j=\frac{N}{p_j}\  and \ r_j=1,\  j=1,2, \  p\geq \max(p_1,p_2),
\end{eqnarray*}
we have the product estimate
\begin{eqnarray*}
  \|uv\|_{\dot{B}_{p,r}^s}
  \leq C(s_1+s_2,\frac{N}{p_1}-s_1,\frac{N}{p_2}-s_2)\|u\|_{\dot{B}_{p_1,r_1}^{s_1}}\|v\|_{\dot{B}_{p_2,r_2}^{s_2}},
\end{eqnarray*}
where
\begin{eqnarray*}
  s-\frac{N}{p}=(s_1-\frac{N}{p_1})+(s_2-\frac{N}{p_2}),\ \ \frac{1}{r}=\min\{1,\frac{1}{r_1}+\frac{1}{r_2}\}.
\end{eqnarray*}
\end{lem}
The following lemma is an application of $s_1+s_2<0$ which frequently emerge in the proof of stability.
\begin{lem}\label{a-theta-product-estimate}
Assume $a\in \dot{B}_{2,1}^{\frac{3}{2}},\ \theta\in
\dot{B}_{2,1}^{\frac{1}{2}}$, then we have
\begin{eqnarray*}\|\nabla(
a\cdot\nabla\theta)\|_{\dot{B}_{2,1}^{-\frac{3}{2}}}\preceq
\|\theta\|_{\dot{B}_{2,1}^{\frac{1}{2}}}\|a\|_{\dot{B}_{2,1}^{\frac{3}{2}}}.\end{eqnarray*}
\end{lem}

\begin{proof}
From Remark \ref{2.2} (3), we have
\begin{eqnarray*}\|\nabla(
a\cdot\nabla\theta)\|_{\dot{B}_{2,1}^{-\frac{3}{2}}}\preceq\|
a\cdot\nabla\theta\|_{\dot{B}_{2,1}^{-\frac{1}{2}}}.
\end{eqnarray*}
Using Remark \ref{2.2}, Lemma \ref{T-estimate} and Lemma \ref{R-estimate}, we get \begin{eqnarray*} \|T_{ a}
\nabla\theta\|_{\dot{B}_{2,1}^{-\frac{1}{2}}}\leq
C\|a\|_{\infty}\|\nabla\theta\|_{\dot{B}_{2,1}^{-\frac{1}{2}}} \leq
C\|a\|_{\dot{B}_{2,1}^{\frac{3}{2}}}\|\nabla\theta\|_{\dot{B}_{2,1}^{-\frac{1}{2}}},
\end{eqnarray*}
\begin{eqnarray*}
\|T_{\nabla \theta} a\|_{\dot{B}_{2,1}^{-\frac{1}{2}}}\leq C\|\nabla \theta\|_{\dot{B}_{\infty,2}^{-2}}\| a\|_{\dot{B}_{2,2}^{\frac{3}{2}}}
\leq C\| \nabla\theta\|_{\dot{B}_{2,1}^{-\frac{1}{2}}}\|a\|_{\dot{B}_{2,1}^{\frac{3}{2}}}
\end{eqnarray*}
and
\begin{eqnarray*}
\|R(a,\nabla\theta)\|_{\dot{B}_{2,1}^{-\frac{1}{2}}}&\leq & C\|R( a,\nabla\theta)\|_{\dot{B}_{1,2}^{1}}\\
&\leq & C\|a\|_{\dot{B}_{2,2}^{\frac{3}{2}}}\|\nabla \theta\|_{\dot{B}_{2,2}^{-\frac{1}{2}}}\\
&\leq & C\|a\|_{\dot{B}_{2,1}^{\frac{3}{2}}}\|\theta\|_{\dot{B}_{2,1}^{\frac{1}{2}}}.
\end{eqnarray*}
Combining the above three inequalities, we can easily get
\begin{eqnarray*}
\|\nabla (a \nabla \theta)\|_{\dot{B}_{2,1}^{-\frac{3}{2}}}
\leq \|a \nabla \theta\|_{\dot{B}_{2,1}^{-\frac{1}{2}}}\preceq \| \theta\|_{\dot{B}_{2,1}^{\frac{1}{2}}}\|a\|_{\dot{B}_{2,1}^{\frac{3}{2}}}.
\end{eqnarray*}
\end{proof}

We give the definition of the Chemin-Lerner type spaces $\tilde{L}_T^\rho(\dot{B}^s_{p,r})$.
\begin{defi}\label{2.3}
For $s\in \mathbb{R}$, $(r,\rho,p)\in[1,+\infty]^3$ , and $T\in(0,+\infty]$, we define
$\tilde{L}_T^\rho(\dot{B}^s_{p,r})$  as the
completion of $C([0,T], S(\mathbb{R}^3))$ under the norm
\begin{eqnarray*}
  \|f\|_{\tilde{L}_T^\rho(\dot{B}^s_{p,r})}
    =\left(\sum_{q\in \mathbb{Z}}2^{qrs}\left(\int_0^T\|\dot{\triangle}_q f(t)\|_{L^p}^\rho dt\right)^{\frac{r}{\rho}}\right)^{\frac{1}{r}}
    <\infty,
\end{eqnarray*}
with the usual change if $r = \infty$.
\end{defi}

\begin{re}\label{2.8}
(i) The space may be linked with the more classical spaces $L_T^\rho(\dot{B}_{p,r}^s)=L^\rho([0.T];\dot{B}_{p,r}^s)$  via the Minkowski's inequality. We have
\begin{eqnarray*}
  \|f\|_{\tilde{L}_T^\rho(\dot{B}_{p,r}^s)}
 \preceq\|f\|_{L_T^\rho(\dot{B}_{p,r}^s)}     \ \  \mbox{if} \ \  \rho\leq r
\end{eqnarray*}
and
\begin{eqnarray*}
  \|f\|_{L_T^\rho(\dot{B}_{p,r}^s)}\preceq \|f\|_{\tilde{L}_T^\rho(\dot{B}_{p,r}^s)}      \ \  \mbox{if} \ \  r\leq \rho.
\end{eqnarray*}

(ii) All the properties of continuity for the product, composition,
remainder and paraproduct of $L_T^\rho(\dot{B}_{p,r}^s)$ may be
easily generalized to the spaces
$\tilde{L}_T^\rho(\dot{B}_{p,r}^s)$. The general principle is that
the time exponent $\rho$ behaves according to H\"{o}lder inequality.
\end{re}

The estimate in $\tilde{L}_T^\rho(\dot{B}^s_{p,r})$ of the following transport equation
\begin{eqnarray}\label{t8}
\left\{
\begin{array}{lll}
   \partial_t f+ v\cdot\nabla f-\Delta f+\lambda\nabla\Pi=g,\\
   f|_{t=0}=f_0
\end{array}
\right.
\end{eqnarray}
will be used in the proof of main results. To derive the estimate for the transport equation (\ref{t8}), we need the following commutator estimate.
\begin{lem}\label{3.2}
Let $r\in [1,\infty], \ f\in \dot{B}_{2,r}^s(\mathbb{R}^3)$, and
$v\in\dot{B}_{2,r}^{\frac{5}{2}}(\mathbb{R}^3)$ with $\mathrm{div}
v=0$. Then there exists a sequence $\{c_{q,r}\}\subset l^r(\mathbb{Z})$ satisfying
$\|c_{q,r}\|_{l^r}=1$ and

(i) If $-\frac{5}{2}<s<\frac{5}{2}$ (or $s=\frac{5}{2}$ with $r=1$),
\begin{eqnarray*}
  \|[\dot{\triangle}_q, v\cdot \nabla]f\|_2\preceq c_{q,r}2^{-sq}\|v\|_{\dot{B}_{2,r}^{\frac{5}{2}}}\|f\|_{\dot{B}_{2,r}^s}.
\end{eqnarray*}

(ii) If $s>-1$,
\begin{eqnarray*}
  \|[\dot{\triangle}_q, v\cdot \nabla]v\|_2\preceq c_{q,r}2^{-sq}\|\nabla v\|_{L^\infty}\|v\|_{\dot{B}_{2,r}^s}.
\end{eqnarray*}
\end{lem}

We give the estimates for the transport equation (\ref{t8}).
\begin{lem}\label{prop:estimate tranport equation}\cite{L-L}
Assume that $s\in[-5/2,5/2]$. Let $v\in
L_T^1(\dot{B}_{2,1}^{5/2}) $ be a divergence-free vector field.
If $\lambda\neq 0$, we also assume that $\mathrm{div} f=0$. For $f_0\in\dot{B}_{2,1}^{s}$ and
$g\in L_T^1(\dot{B}_{2,1}^{s})$, let $f\in
L_T^1(\dot{B}_{2,1}^{s+2})$  and $\nabla\Pi\in
L_T^1(\dot{B}_{2,1}^{s})$ solve (\ref{t8}). Then we have, for $t\in (0,T)$
\begin{eqnarray} \label{t9}
  &&\|f\|_{\tilde{L}_t^\infty(\dot{B}_{2,1}^{s})}+\|f\|_{L_t^1(\dot{B}_{2,1}^{s+2})}
     +\lambda\|\nabla\Pi\|_{L_t^1(\dot{B}_{2,1}^{s})}\nonumber\\
  &&\leq \|f_0\|_{\dot{B}_{2,1}^s}+C(\int_0^t
    \| v(t')\|_{\dot{B}_{2,1}^{\frac{5}{2}}}\|f\|_{\tilde{L}^\infty_{t'}(\dot{B}_{2,1}^{s})}dt')
    +\|g\|_{L^1_t(\dot{B}_{2,1}^s)}.
\end{eqnarray}
In the case $f=v$ and $-1<s<5/2$, we have, for $t\in (0,T)$
\begin{eqnarray}\label{t10}
  && \|f\|_{\tilde{L}_t^\infty(\dot{B}_{2,1}^s)}+\|f\|_{L_t^1(\dot{B}_{2,1}^{s+2})}+\|\nabla\Pi\|_{L_t^1(\dot{B}_{2,1}^{s})}\nonumber\\
  && \leq\|f_0\|_{\dot{B}_{2,1}^s}+C(\int_0^t\|\nabla
     v(t')\|_{L^\infty}\|f\|_{\tilde{L}^\infty_{t'}(\dot{B}_{2,1}^{s})}dt')+\|g\|_{L^1_t(\dot{B}_{2,1}^s)}.
\end{eqnarray}
\end{lem}

\section{Existence and decay of weak solutions}
\setcounter{equation}{0}

In this section we construct the weak solution for (\ref{rewrite-boussinesq}) with $\kappa=1$,  $\nu=0$. We begin by introducing a mollified Boussinesq system constructed as in \cite{BS}. We first recall the definition of the ``retarded mollifier". Let $\psi(x,t)\in C^\infty(\mathbb{R}^3\times\mathbb{R})$ such that
\begin{eqnarray*}
\psi\geq 0,\quad\int_0^\infty\int \psi dxdt=1,\quad {\rm supp} \psi \subset \{(x,t): |x|^2<t, 1<t<2\}.
\end{eqnarray*}
For $T>0$ and $u\in L^2(0,T; L^2_{\boldsymbol{\sigma}})$, let
$\bar{u}:\mathbb{R}^3\times\mathbb{R}\rightarrow\mathbb{R}^3$ be
\begin{eqnarray*}
    \bar{u}=
    \left\{
    \begin{array}{lll}
    u(x,t)\  \  \  \  && \mathrm{if} \ (x,t)\in \mathbb{R}^3\times (0,T),\\
    0 \ \ \    \ && \mathrm{otherwise}.
    \end{array}
    \right.
\end{eqnarray*}
Let $\delta=T/n$. We set
\begin{eqnarray*}
    \Psi_\delta(u)(x,t)=\delta^{-4}\int_{\mathbb{R}^4}\psi\left(\frac{y}{\delta},\frac{\tau}{\delta}\right)\bar{u}(x-y,t-\tau)dyd\tau.
\end{eqnarray*}
Consider, for $n = 1, 2, \cdots$ and $\delta = T/n$, the mollified
Cauchy problem
\begin{eqnarray} \left\{\begin{array}{lll}
 \medskip
\partial_t a^n + \Psi_\delta (u^{n-1})\cdot\nabla a^n=0, \qquad (t,x)\in \mathbb{R^+}\times\mathbb{R}^3,\\
 \medskip
\partial_t \theta^n+ \Psi_\delta (u^{n-1})\cdot\nabla \theta^n-\Delta\theta^n+\nabla (a^n\nabla\theta^n)=0,\\
 \medskip
\partial_t u^{n}+ \nabla(\Psi_\delta (u^{n})\otimes u^n)-(1+a^n) \Delta u^n+ (1+a^n)\nabla\Pi^n=\theta^n e_3,\\
\medskip
\mathrm{div}u^n=0
 \end{array}\right.\label{mollified-boussinesq}
 \end{eqnarray}
with initial data
\begin{eqnarray*}
  a^n|_{t=0}=a_0,\ \ \ \theta^n|_{t=0}=\theta_0\ \  \ \ and \ \  \ \ u^n|_{t=0}=u_0.
\end{eqnarray*}
The iteration scheme starts with $u^0 = 0$. Note that since ${\rm
div} u = 0$, we also have $\mathrm{div}(\Psi_\delta(u^{n})) = 0$,
for $t\in \mathbb{R}^+$. At each step $n$, one solves recursively
$n+1$ linear equations: first one solves the transport-diffusion
equation (with smooth convective velocity) for the density and
temperature; after $a^n, \theta^n$ is computed, solving the third part
of (\ref{mollified-boussinesq}) amounts to solving a linear equation
on each strip $\mathbb{R}^3 \times (m\delta, (m+1)\delta)$, for $m
= 0, 1, \cdots , n-1$. So we can get the existence of the mollified
Boussinesq system (\ref{mollified-boussinesq}) like the Proposition
3.1 in \cite{BS}. Then there is no additional difficulty in the
passage to the limit in the non-linear terms arises in the equation
of the temperature other than those already existing for the
Navier-Stokes equations. Hence, the distributional limit $(a,\theta,
u, \Pi)$ of a convergent subsequence of $(a^n, \theta^n, u^n,
\Pi^n)$ is a weak solution of the Boussinesq system.

As a consequence of the above arguments, we obtain the existence of weak solution to system (\ref{rewrite-boussinesq}) with $\kappa=1$, $\nu=0$. The results of following Proposition are established for (\ref{mollified-boussinesq}). For the sake of convenience, we simply denote the solutions $(a^n; \theta^n; u^n; \Pi^n)$ by $(a; \theta; u; \Pi)$.
\begin{prop}\label{1.1}
Let $m\leq a_0, \ (a_0,\theta_0, u_0)\in{B}_{2,1}^{\frac{3}{2}}\times L^2\times L^2_{\sigma}$, and there exist two absolute constants $\eta_0>0, \ C_1>0$  such that if
\begin{eqnarray}
\|a_0\|_{\dot{B}_{2,1}^{\frac{3}{2}}}<\eta_0
\label{a-satisfy}\end{eqnarray}
and
\begin{eqnarray}\label{u-satisfy}
 \int_0^\infty\|\nabla u\|_{\infty}d\tau<M_1,
\end{eqnarray}
then there exists a weak solution $(a, \theta, u)$ of the Boussinesq system
(\ref{rewrite-boussinesq}) for $\kappa=1,\nu=0$ with initial data
$(a_0, \theta_0, u_0)$, such that for any $T>0$,
\begin{eqnarray*}
&&a\in L^\infty(0,T;L^\infty)\cap L^\infty(0,T;W^{1,3})\cap L^\infty(0,T;L^2),\\
&&\theta\in L^2(0,T; H^1)\cap L^\infty(0,T;L^2_{\sigma}),\\
&&u\in L^2(0,T; V)\cap L^\infty(0,T;L^2_{\sigma}).
\end{eqnarray*}
Moreover, such solution satisfies, for all $t \in [0, T]$,  the energy inequalities
\begin{eqnarray*}
&&m\leq a \leq M,\ \|a\|_{2}\leq C\|a_0\|_{2};\\
&&\|a(t)\|_{W^{1,3}}\leq C(\eta_0+\eta_0^{\frac{1}{3}}\|a_0\|_2^{\frac{2}{3}});\\
&&\|\theta(t)\|_2^2+2c_0\int_0^t\|\nabla\theta(s)\|_2^2 \mathrm{d}s\leq\|\theta_0\|_2^2;\\
&&\|u(t)\|_2^2+{2(1+M)}\int_0^t\|\nabla u(s)\|_2^2\mathrm{d}s\leq C(\|u_0\|_2^2+t^2\|\theta_0\|_2^2),
\end{eqnarray*}
for all $t\geq 0$ and some constant $ c_0,\ C>0,\ M=C\eta_0$.
\end{prop}

\begin{proof}
We only prove the energy inequalities. According to the transport equation in
(\ref{rewrite-boussinesq}) and $\mathrm{div}u=0$, we have
\begin{eqnarray} \|a\|_2\leq C \|a_0\|_2;\ m\leq a\leq\|a\|_\infty \leq C\|a_0\|_\infty\leq C\|
a_0\|_{\dot{B}_{2,1}^{\frac{3}{2}}}\leq C\eta_0, \label{a1}\end{eqnarray}
\begin{eqnarray*}
\|\nabla a\|_{3}\leq \|
\nabla a_0\|_{3}+C\int_0^t \|\nabla
u\|_{\infty}\|
\nabla a\|_{3}d\tau. \end{eqnarray*}
 Applying the Gronwall
inequality and (\ref{u-satisfy}) to the above inequality implies
\begin{eqnarray}\|\nabla a\|_{3}\leq C\|\nabla a_0\|_3\exp\{\int_0^\infty\|\nabla u\|_\infty dt\}\leq C\|
 a_0\|_{\dot{B}_{2,1}^{\frac{3}{2}}}\leq C\eta_0, \label{a}\end{eqnarray}
here we
use the imbedding of $\dot{B}_{2,1}^{\frac{3}{2}}\hookrightarrow
\dot{B}_{3,1}^1\hookrightarrow L^\infty$, and $\|a_0\|_{\dot{B}_{2,1}^{\frac{3}{2}}}<\eta_0.$
From interpolation equality in Sobolev space, we can get
\begin{eqnarray*} \|a_0\|_3\leq \|a_0\|_2^{\frac{2}{3}}\|a_0\|_\infty^{\frac{1}{3}}\leq C\eta_0^{\frac{1}{3}}\|a_0\|_2^{\frac{2}{3}},\end{eqnarray*} the above inequalities completed the proof of the energy inequalities about $a$. We get by a standard energy estimate to temperature equation of (\ref{rewrite-boussinesq})
that
\begin{eqnarray*} \frac{1}{2}\frac{d}{dt}\|\theta\|_2^2+\|\nabla\theta\|_2^2=
(\nabla (a\nabla\theta), \theta) \leq \|
a\|_\infty\|\nabla\theta\|_2^2\leq C\eta_0\|\nabla\theta\|_2^2 ,\end{eqnarray*}
we have \begin{eqnarray*}
\frac{1}{2}\frac{d}{dt}\|\theta\|_2^2+\|\nabla\theta\|_2^2 \leq
C\eta_0\|\nabla\theta\|_2^2 \label{basic-estimate},\end{eqnarray*}
 and easily get
\begin{eqnarray} \frac{1}{2}\frac{d}{dt}\|\theta\|_2^2+c_0\|\nabla\theta\|_2^2
\leq 0;\label{basic-estimate}
\end{eqnarray}
As $a=\frac{1}{\rho}-1$, a direct computation yields
\begin{eqnarray}\label{basic-estimate-u}
\frac{1}{2}\frac{d}{dt}\left\|\frac{u}{\sqrt{1+a}}\right\|_2^2+\|\nabla u\|_2^2=\int\frac{\theta e_3 u }{1+a}dx\leq
\frac{1}{\sqrt{1+m}}\|\theta\|_2\left\|\frac{u}{\sqrt{1+a}}\right\|_2,
\end{eqnarray}
which along with the free transport equation in (\ref{sys:Boussinesq}) gives
\begin{eqnarray*}\frac{1}{2}\|\theta\|_2^2+c_0\int_0^t\|\nabla\theta\|_2^2dt'\leq \frac{1}{2}\|\theta_0\|_2^2,\end{eqnarray*}
\begin{eqnarray*}\frac{1}{2}\left\|\frac{u}{\sqrt{1+a}}\right\|_2^2+\int_0^t\|\nabla u\|_2^2dt'\leq \frac{\|u_0\|_2^2}{1+m}+\frac{t^2}{1+m}\|\theta_0\|_2^2.\end{eqnarray*}
This together with (\ref{a}) completes the proof of Proposition \ref{1.1}.
\end{proof}
\begin{re}
The conditions of above Proposition and Lemma \ref{u-v-tidu} about $a$ can be weakened to $a_0\in \dot{B}_{3,1}^1\cap L^2$, and $\|a_0\|_{\dot{B}_{3,1}^1}\leq \eta_0$. we can get same conclusion.
\end{re}

\section{Proof of Theorem \ref{main-decay-bounded}}
In the section we establish the decay estimate for the weak solution constructed in Proposition \ref{1.1} to complete the proof of Theorem \ref{main-decay-bounded}. For the solutions to (\ref{mollified-boussinesq}) we have the following $L^2$ decay estimates.
\begin{lem}\label{3.12}
Let $(a_0,\theta_0,u_0)$ is chosen in Proposition \ref{1.1} and $\theta_0\in L^1$. Then the solution $(a^n, \theta^n,
u^n)$ of the mollified system(\ref{mollified-boussinesq}) for some $n\in {1,2,\cdots}$ satisfies
\begin{eqnarray}\label{first-decay-estimate}
   \|\theta^n(t)\|_2 \leq \|\theta_0\|_1(Ct+A)^{-\frac{3}{4}}
\end{eqnarray}
and
\begin{eqnarray}\label{u-first-decay-estimate}
   \|u^n(t)\|_2 \leq \frac{\sqrt{1+M}}{\sqrt{1+m}}\|u_0\|_2+C' \|\theta_0\|_1
   t^{\frac{1}{4}}
\end{eqnarray}
for two absolute constants $C,C'> 0$ and $A=A(\|\theta_0\|_1,\|\theta_0\|_2)$.
\end{lem}
\begin{proof}
This can be justified by testing the equation in (\ref{mollified-boussinesq}) for $\theta^n$ by $\theta^n$ and the velocity equation by $u^n$ (see \cite[Lemma 3.3]{BS} for details).
\end{proof}

In the following lemma we get the ordinary differential inequalities concerning $u^n$ and $\theta^n$.
\begin{lem}\label{u-v-tidu}
Let $(\theta_0, u_0)\in H^1\times H^1$. (\ref{a-satisfy}) and (\ref{u-satisfy}) satisfied.
 Assume
\begin{eqnarray*}
\int_0^t\|\nabla u^{n-1}\|_{2}^2d\tau<C,\quad \|u^{n-1}\|_2 \leq C_0.
\end{eqnarray*}
Under the additional assumptions $\theta_0\in L^1$, if for some $\varepsilon_0 > 0$,
\begin{eqnarray}\label{a5}
 \|\theta_0\|_{L^1}<\varepsilon_0,
\end{eqnarray}
then there exists $t_0>0$ such that the solutions $(a^n, \theta^n, u^n)$ of the mollified system (\ref{mollified-boussinesq}) satisfies
\begin{eqnarray}\label{e2}
  \frac{d}{dt}\|\nabla u^{n-1}(t)\|_2^2+\frac{m+1}{2C}\|\Delta u^{n-1}(t)\|_2^2\leq C\|\theta^{n-1}(t)\|_2^2, \ \ t>t_0
\end{eqnarray}
and
\begin{eqnarray}\label{e1}
  \frac{d}{dt}\|\nabla \theta^n(t)\|_2^2+\|\Delta \theta^n(t)\|_2^2\leq 0, \ \ t>t_0.
\end{eqnarray}
\end{lem}
\begin{proof}
Taking the $L^2$ inner product of the velocity equations in (\ref{mollified-boussinesq}) with $\Delta u^{n-1}$, applying Young's inequality, we have
\begin{eqnarray}
   && \frac{1}{2}\frac{d}{dt}\|\nabla u^{n-1}\|_2^2+(1+m)\|\Delta u^{n-1}\|_2^2\nonumber\\
   && \leq \frac{1}{2}\frac{d}{dt}\|\nabla u^{n-1}\|_2^2+\|\sqrt{1+a^{n-1}}\Delta u^{n-1}\|_2^2\nonumber\\
   && =-( \Psi_\delta(u^{n-1})\cdot\nabla u^{n-1}|\Delta u^{n-1})_2-((1+a^{n-1})\nabla\Pi^{n-1}| \Delta u^{n-1})_2+(\theta^{n-1} |\Delta u^{n-1})_2\nonumber\\
   && \leq C \|u^{n-1}\|_3\|\nabla u^{n-1}\|_6\|\Delta u^{n-1}\|_2+\|1+a^{n-1}\|_\infty\|\nabla\Pi^{n-1}\|_2\|\Delta u^{n-1}\|_2+\|\theta^{n-1}\|_2
       \|\Delta u^{n-1}\|_2\nonumber\\
   &&\leq C \| u^{n-1}\|_2^{1/2}\|\nabla u^{n-1}\|_2^{1/2}\|\Delta u^{n-1}\|_2^2+(1+M)\left(\frac{\|\nabla\Pi^{n-1}\|_2+\|\Delta u^{n-1}\|_2}{2}\right)^2\nonumber\\
   && \quad+C\|\theta^{n-1}\|_2^2+\frac{1+m}{8}\|\Delta u^{n-1}\|_2^2.
\label{u}\end{eqnarray}
Again thanks to Young's inequality and $\mathrm{div} u^{n}=0$, we obtain
\begin{eqnarray}\label{pressure}
&&\|\Delta u^{n-1}\|_2+\|\nabla\Pi^{n-1}\|_2\nonumber\\
&&\leq \sqrt{2}\|\Delta u^{n-1}-\nabla\Pi^{n-1}\|_2\nonumber\\
&&\leq \sqrt{2}\|\frac{\partial_t u^{n-1}}{1+a^{n-1}}+\frac{ \nabla(\Psi_\delta (u^{n-1})\otimes u^{n-1})}{1+a^{n-1}}-\frac{\theta^{n-1}e_3}{1+a^{n-1}}\|_2\nonumber\\
&&\leq C\|\frac{\partial_t u^{n-1}}{\sqrt{1+a^{n-1}}}\|_2+C\|u^{n-1}\|_3\|\Delta u^{n-1}\|_2+C\|\theta^{n-1}\|_2\nonumber\\
&&\leq C\|\frac{\partial_t u^{n-1}}{\sqrt{1+a^{n-1}}}\|_2+C\|u^{n-1}\|_2^{1/2}\|\nabla u^{n-1}\|_2^{1/2}\|\Delta u^{n-1}\|_2+C\|\theta^{n-1}\|_2,
\end{eqnarray}
as a consequence, we have
\begin{eqnarray*}
&&\left(\frac{\|\Delta u^{n-1}\|_2+\|\nabla\Pi^{n-1}\|_2}{2}\right)^2\\
&&\leq C\|\frac{\partial_t u^{n-1}}{\sqrt{1+a^{n-1}}}\|_2^2+C\|u^{n-1}\|_2\|\nabla u^{n-1}\|_2\|\Delta u^{n-1}\|_2^2+C\|\theta^{n-1}\|_2^2\\
&&\leq C\|\frac{\partial_t
u^{n-1}}{\sqrt{1+a^{n-1}}}\|_2^2+\left(\frac{1+m}{4(1+M)}+C_2^*\|\nabla
u^{n-1}\|_2\right)\|\Delta u^{n-1}\|_2^2+C\|\theta^{n-1}\|_2.
\end{eqnarray*}
Putting the above inequality into (\ref{u}), we get
\begin{eqnarray}\label{e9}
   &&\frac{d}{dt}\|\nabla u^{n-1}\|_2^2+(1+m-C_2\|\nabla u^{n-1}\|_2-C_1\|\nabla u^{n-1}\|_2^{1/2})\|\Delta u^{n-1}\|_2^2\nonumber\\
   &&\leq C\|\theta^{n-1}\|_2^2+C_3\|\frac{\partial_t u^{n-1}}{\sqrt{1+a^{n-1}}}\|_2^2,
\end{eqnarray}
where $C_1=CC_0^{1/2},\ C_2=(1+M)C_2^{*}, \ C_3$ are constants. On the other hand, taking the $L^2$ inner product of the velocity equation in (\ref{mollified-boussinesq}) with $\frac{\partial_t u^{n-1}}{1+a^{n-1}}$, it follows from Young's inequality and using integration by parts to conclude
\begin{eqnarray*}
&&\left\|\frac{\partial_t u^{n-1}}{\sqrt{1+a^{n-1}}}\right\|_2^2+\frac{1}{2}\frac{d}{dt}\|\nabla u^{n-1}\|_2^2\\
&&=-\left(\frac{\nabla(\Psi_\delta (u^{n-1})\otimes u^{n-1})}{1+a^{n-1}}|\partial_t u^{n-1}\right)_{L^2}+\left(\frac{\theta^{n-1} e_3}{1+a^{n-1}}|\partial_t u^{n-1}\right)_{L^2}\\
&&\leq\left\|\frac{1}{\sqrt{1+a^{n-1}}}\right\|_\infty\|u^{n-1}\|_3\|\nabla u^{n-1}\|_6\left\|\frac{\partial_t u^{n-1}}{\sqrt{1+a^{n-1}}}\right\|_2+\|\sqrt{1+a^{n-1}}\|_\infty\|\theta^{n-1}\|_2\left\|\frac{\partial_t u^{n-1}}{\sqrt{1+a^{n-1}}}\right\|_2\\
&&\leq C\|u^{n-1}\|_2^{\frac{1}{2}}\|\nabla u^{n-1}\|_2^{\frac{1}{2}}\|\Delta u^{n-1}\|_2\left\|\frac{\partial_t u^{n-1}}{\sqrt{1+a^{n-1}}}\right\|_2+C\|\theta^{n-1}\|_2^2+\frac{1}{4}\left\|\frac{\partial_t u^{n-1}}{\sqrt{1+a^{n-1}}}\right\|_2^2\\
&&\leq C\|\nabla u^{n-1}\|_2\|\Delta u\|_2^2+C\|\theta^{n-1}\|_2^2+\frac{1}{2}\left\|\frac{\partial_t u^{n-1}}{\sqrt{1+a^{n-1}}}\right\|_2^2.
\end{eqnarray*}
Therefore, we obtain
\begin{eqnarray*}
\frac{d}{dt}\|\nabla u^{n-1}\|_2^2+\left\|\frac{\partial_t u^{n-1}}{\sqrt{1+a^{n-1}}}\right\|_2^2\leq C(\|\theta^{n-1}\|_2^2+\|\nabla u^{n-1}\|_2\|\Delta u^{n-1}\|_2^2).
\end{eqnarray*}
This along with (\ref{e9}) ensures a positive constant $e_1$ such that
\begin{eqnarray*}
 &&\frac{d}{dt}\|\nabla u^{n-1}\|_2^2+e_1\|\partial_t u^{n-1}\|_2^2\\
 &&\quad+\left(\frac{1+m}{2C_4}-\frac{C_1}{2C_4}\|\nabla u^{n-1}\|_2^{\frac{1}{2}}-\frac{C_2}{2C_4}\|\nabla u^{n-1}\|_2-C_4\|\nabla u^{n-1}\|_2\right)\|\Delta u^{n-1}\|_2^2 \\
 &&\leq C\|\theta^{n-1}\|_2^2,
\end{eqnarray*}
where $C_4$ is a constant dependent on $C_3$. Due to the condition $\int_0^t\|\nabla u^{n-1}\|_2^2 d\tau<C$, for any $\eta>0$, there exists $t=t_0(\eta)>0$ such that
$$ \|\nabla u^{n-1}(t_0)\|_2\leq \eta.$$
Now choosing $1>\eta>0$, for example $\eta\leq\frac{(1+m)^2}{36M^{*2}}$ such that
\begin{eqnarray*}
\eta^{1/2}+2\eta\leq\frac{1+m}{2M^*},\quad M^*=\max\{C_1,C_2,2C_4^2\}
\end{eqnarray*}
Define
\begin{eqnarray}\label{e14}
   t^*=\sup\{t\geq t_0, \|\nabla u^{n-1}(t)\|_2\leq 2\eta\},
\end{eqnarray}
and we claim that $t^*=\infty$. Indeed, if $t^*<\infty$, Lemma \ref{3.12} and $\|\theta_0\|_{L^1}<\varepsilon_0$ imply that
\begin{eqnarray*}
   \frac{d}{dt}\|\nabla u^{n-1}\|_2^2+e_1\|\partial_t u^{n-1}\|_2^2+\frac{1+m}{2C}\|\Delta u^{n-1}\|_2^2&\leq& C\|\theta^{n-1}\|_2^2\\
   &\leq& C\|\theta_0\|_{L^1}^2(1+t)^{-\frac{3}{2}}\\
   &\leq& C\varepsilon_0^2(1+t)^{-\frac{3}{2}},
\end{eqnarray*}
for $t\in[t_0, t^*]$, which gives
\begin{eqnarray*}
   && \|\nabla u^{n-1}(t^*)\|_2^2+ \frac{1+m}{2C}\int_{t_0}^{t^*}\|\Delta u^{n-1}(t')\|_2^2 dt'\\
   && \leq \eta^2+C\int_{t_0}^{t^*}\varepsilon_0^2(1+t')^{-\frac{3}{2}}dt'
   \leq \eta^2+\frac{C\varepsilon_0^2}{(1+t_0)^{\frac{1}{2}}} \quad\mathrm{for}\ t\in[t_0, t^*].
\end{eqnarray*}
If we choose $\varepsilon_0\leq
\frac{\eta^{1/2}(1+t_0)^{\frac{1}{4}}}{C^{1/2}}$ such that
\begin{eqnarray*}
  2\eta>\eta^2+\frac{2C\varepsilon_0^2}{(1+t_0)^{\frac{1}{2}}},
\end{eqnarray*}
we find that this contradicts (\ref{e14}), and thus $t^*=\infty$. Then the above inequality together with (\ref{e14})
concludes the proof of the proposition about velocity.

Following the same line, taking the $L^2$ inner product of the temperature equations of (\ref{mollified-boussinesq}) with $\Delta \theta^{n}$ and applying (\ref{a1})-(\ref{a}), we get
\begin{eqnarray*}
   \frac{1}{2}\frac{d}{dt}\|\nabla\theta^{n}\|_2^2+\|\Delta \theta^{n}\|_2^2&=&-(\Psi_\delta(u^{n-1})\cdot\nabla \theta^{n}|\Delta \theta^{n})_2+(\nabla (a^n\nabla\theta^n),\Delta\theta^n)
   \\&=&-(\Psi_\delta(u^{n-1})\cdot\nabla \theta^{n}|\Delta \theta^{n})_2+(\nabla a^n\nabla\theta^n,\Delta\theta^n)+(a^n\Delta\theta^n,\Delta\theta^n)
   \\&\leq &C\|u^{n-1}\|_3\|\nabla\theta^{n}\|_6\|\Delta\theta^{n}\|_2+\|\nabla a^n\|_3\|\Delta\theta^n\|_2^2+\|a^n\|_\infty\|\Delta\theta^n\|_2^2
   \\&\leq &C_1(\eta^{\frac{1}{2}}+2\eta_0)\|\Delta\theta^{n}\|_2^2,
\end{eqnarray*}
where $\eta$ is defined as above. As a consequence, we obtain
\begin{eqnarray*}
   \frac{d}{dt}\|\nabla \theta^{n}\|_2^2+\|\Delta\theta^n\|_2^2\leq 0.
\end{eqnarray*}
This completes the proof of Lemma \ref{u-v-tidu}.
\end{proof}
A combination of Lemma \ref{3.12} and Lemma \ref{u-v-tidu} yields the following decay estimate of gradient $\nabla\theta^n$.
\begin{lem}\label{3.9}
Under assumptions of Lemmas \ref{3.12} and \ref{u-v-tidu}, there holds
\begin{eqnarray}\label{l20}
  \|\nabla\theta^n(t)\|_2^2\leq C (1+t)^{-\frac{5}{2}}.
\end{eqnarray}
\end{lem}
\begin{proof}
We apply the Fourier splitting method introduced in \cite{Sch}. First using the
Plancherel theorem in the energy inequality (\ref{e1}) for $\theta^n$
and splitting the integral in $\mathbb{R}^3$ into $S\cup S^c$, where
$S(t)=\{\xi: |\xi|\leq g(t)\}$ with $g(t)
=(\frac{4}{1+t})^{\frac{1}{2}}$, we get
\begin{eqnarray*}
  \frac{d}{dt}\|\nabla\theta^n\|_2^2+ g^2\|\nabla\theta^n\|_2^2\preceq g^4 \int_{S(t)} |\hat{\theta^n}|^2d\xi\preceq g^4 t^{-3/2}\preceq t^{-\frac{5}{2}-1},
\end{eqnarray*}
it follows that
\begin{eqnarray*}
   \|\nabla\theta^n\|_2\preceq(1+t)^{-5/4}.
\end{eqnarray*}
This completes the proof of Lemma \ref{3.9}.
\end{proof}

It is natural to ask under which supplementary conditions on the
initial data one can ensure that the energy of the fluid $\|u^n\|_2$
remains uniformly bounded. The following lemma gives the answer.
\begin{lem}\label{3.5}
Assume that $(a_0,\theta_0, u_0)$ satisfy the condition of
Theorem \ref{main-decay-bounded}, then the solutions $(a^n,
\theta^n, u^n)$ of the mollified Boussinesq system
(\ref{mollified-boussinesq}) satisfies
\begin{eqnarray}\label{c18}
  \|\theta^n(t)\|_2^2\leq C (1+t)^{-\frac{5}{2}},
 \end{eqnarray}
 \begin{eqnarray}\label{u}
     \|u^n(t)\|_2^2\leq C,
\end{eqnarray}
  \begin{eqnarray}\label{u-bounded}
    \|u^n\|_2^2 + 2(1+M)\int_0^t \|\nabla u^n(t')\|_2^2 dt'
   \leq C
\end{eqnarray}
for some constant $C>0$ and $t>0$.
\end{lem}
\begin{proof}
We denote by C a positive absolute constant, which may change from line
to line. We also denote by A1, A2,... positive constants that depend only on the
data. Here we need to use a bootstrap argument to get the decay estimate of $\theta^n$.
\vskip 3mm
Step 1. An auxiliary estimate

We write the temperature  equation of (\ref{mollified-boussinesq}) as the following integral equation
\begin{eqnarray}
\theta^n(t)=e^{t\Delta}\theta_0+\int_0^te^{(t-s)\Delta}\Psi_\delta(u^{n-1})\nabla
\theta^n ds+\int_0^te^{(t-s)\Delta}\nabla(a^n\nabla\theta^n)ds, \end{eqnarray}
\label{solution-theta-n}
where
\begin{eqnarray*}
e^{t\Delta}f(x)=\int_{R^3}E_t(x-y)f(y)dy,\ E_t(x)=\frac{1}{(4\pi t)^{\frac{3}{2}}}e^{-\frac{|x|^2}{4t}}.
\end{eqnarray*}
We follow the Fourier splitting method introduced in \cite{BS,Sch}. Using the Plancherel theorem in the energy inequality for $\theta^n$, we get
\begin{eqnarray*}
\frac{1}{2}\frac{d}{dt}\int|\widehat{\theta^n(\xi,t)}|^2d\xi \leq -c_0\int|\xi|^2|\widehat{\theta^n(\xi,t)}|^2d\xi,
\end{eqnarray*}
then split the integral on the right-hand side into $S\cup S^c$, where
$$
S=\{\xi:|\xi|\leq(\frac{k}{2c_0(t+1)})^{\frac{1}{2}} \}
$$
it follows that
\begin{eqnarray*}
\frac{1}{2}\frac{d}{dt}\int|\widehat{\theta^n(\xi,t)}|^2d\xi +\frac{k}{1+t}\int|\widehat{\theta^n(\xi,t)}|^2d\xi\leq \frac{k}{1+t}\int_S|\widehat{\theta^n(\xi,t)}|^2d\xi,
\end{eqnarray*}
multiplying by $(1+t)^k$ we obtain
\begin{eqnarray}
\frac{d}{dt}[\int(1+t)^k|\widehat{\theta^n(\xi,t)}|^2d\xi]\leq k(1+t)^{k-1}\int_S|\widehat{\theta^n(\xi,t)}|^2d\xi.
\label{fourier-split}\end{eqnarray}
Because
\begin{eqnarray*}
|\widehat{\theta^n(\xi,t)}|^2&\leq &2|e^{t\Delta}\theta_0|^2+2|\xi|^2(\int_0^t\|u^{n-1}\|_2
\|\theta^n\|_2ds)^2
+2|\xi|^2(\int_0^t\|a^n\|_2\|\nabla\theta^n\|_2ds)^2.
\end{eqnarray*}
Replacing this in (\ref{fourier-split}) and applying the Plancherel theorem, we can get
\begin{eqnarray*}
&&\frac{d}{dt}\left[\int(1+t)^k\|\theta^n(\xi,t)\|_2^2d\xi\right]\\
&&\leq C\left[\|e^{t\Delta}\theta_0\|_2^2(1+t)^{k-1}+(1+t)^{k-\frac{7}{2}}(\int_0^t\|u^{n-1}\|_2
\|\theta^n\|_2ds)^2
+(1+t)^{k-\frac{7}{2}}(\int_0^t\|a^n\|_2\|\nabla\theta^n\|_2ds)^2\right].
\end{eqnarray*}
As $\int \theta_0dx=0$, $\theta_0\in L_1^1$, we easily get
\begin{eqnarray*} \|e^{t\Delta}\theta_0\|^2\leq A_1(1+t)^{-5/2}.\end{eqnarray*}
 Let $k = 7/2$. For all $n\in N$, it follows that,
\begin{eqnarray}
\|\theta^n(t)\|_2&\leq & C\Big[A_1(1+t)^{-\frac{5}{2}}+(1+t)^{-\frac{7}{2}}\int_0^t(\int_0^s\|u^{n-1}(r)\|_2
\|\theta^n(r)\|_2dr)^2ds
\nonumber\\
&&+(1+t)^{-\frac{7}{2}}\int_0^t(\int_0^t\|a^n\|_2\|\nabla\theta^n\|_2ds)^2\Big].
\label{temperature-fourier}\end{eqnarray}

Step 2. The inductive argument.

Consider the following induction hypothesis:
\begin{eqnarray}\label{u-n-1-estimate}
\|u^{n-1}\|_2\leq C,\quad \int_0^t\|\nabla u^{n-1}\|_2^2\leq C.
\end{eqnarray}
For $n=1$ the inductive condition (\ref{u-n-1-estimate}) is immediate since $u^0= 0$.  Let us now prove that $\|u^{n}\|_2\leq C$ and $\int_0^t\|\nabla u^{n}\|_2^2\leq C$, assuming that (\ref{u-n-1-estimate}) holds true. From (\ref{first-decay-estimate}), (\ref{l20}) and Proposition \ref{1.1}, we have the following estimates,
 \begin{eqnarray*}
  \|\theta^n\|_2\leq C\|\theta_0\|_1t^{-\frac{3}{4}},\quad \|a^n\|_2\leq C\|a_0\|_2,\quad
\|\nabla\theta^n\|_2\leq Ct^{-\frac{5}{4}}.
\end{eqnarray*}
Putting these inside (\ref{temperature-fourier}), we can obtain
\begin{eqnarray}
\|\theta^n(t)\|_2&\leq & C\Big[A_1(1+t)^{-\frac{5}{2}}+(1+t)^{-\frac{7}{2}}\|\theta_0\|_2\int_0^t(\int_0^s\|u^{n-1}(r)\|_2
r^{-\frac{3}{4}}dr)^2ds
\nonumber\\
&&+(1+t)^{-\frac{7}{2}}\|a_0\|_2\int_0^t(\int_0^sr^{-\frac{5}{4}}dr)^2ds\Big].
\label{temperature-fourier-1}\end{eqnarray}
Then we can get
\begin{eqnarray*}
\|\theta^n(t)\|_2^2\leq (1+t)^{-2}.
\end{eqnarray*}
Combining the above inequality and following the method of Lemma \ref{3.9}, we have
\begin{eqnarray*}
\|\nabla\theta^n\|_2\leq C(1+t)^{-\frac{3}{2}}.
\end{eqnarray*}
Inserting the above two inequalities inside (\ref{temperature-fourier}) again, we can obtain
\begin{eqnarray*}
\|\theta^n\|_2\leq C(1+t)^{-\frac{5}{4}}.
\end{eqnarray*}
Applying the method of Lemma \ref{3.9} again, we have
\begin{eqnarray}\label{n-th}
\|\nabla\theta^n\|_2\leq C(1+t)^{-\frac{7}{4}}.
\end{eqnarray}
\vskip 3mm
Step 3. Uniform bound for the $L^2$-norm of the velocities $u^n$ and  $\int_0^t\|\nabla u^{n}\|_2^2\leq C$.

Back to the differential inequality (\ref{basic-estimate-u}), we can get
\begin{eqnarray*}\frac{d}{dt}\left\|\frac{u^n}{\sqrt{1+a^n}}\right\|_2\leq
\frac{1}{\sqrt{1+m}}\|\theta^n\|_2, \end{eqnarray*}
As $$
\|\theta^n\|_2\leq C(1+t)^{-\frac{5}{4}},
$$ we finally get
\begin{eqnarray*}\left\|u^n\right\|_2\leq \left\|u_0\right\|_2+C\int_0^t(1+s)^{-\frac{5}{4}}ds\leq C\end{eqnarray*}
Employing the differential inequality (\ref{basic-estimate-u}) again, we can get
 \begin{eqnarray*}
    \frac{1}{2(1+M)}\|u^n\|_2^2 + \int_0^t \|\nabla u^n(s)\|_2^2 ds
    &\leq&
    \frac{1}{2}\|\frac{u^n}{\sqrt{1+a^n}}\|_2^2 + \int_0^t \|\nabla u^n(s)\|_2^2 ds
    \\ &\leq &\frac{1}{2}\|\frac{u_0}{\sqrt{1+a_0}}\|_2^2+\int_0^t\int \frac{\theta^n u^ne_3 }{\sqrt{1+a_0}} dxds\\
    \\&\leq &\frac{1}{2}\|\frac{u_0}{\sqrt{1+a_0}}\|_2^2+C\int_0^t\|\theta^n\|_{2}\|u^n\|_{2}ds \\
    &\leq &\frac{1}{2(1+m)}\|u_0\|_2^2+\int_0^t C(1+s)^{-\frac{5}{4}}ds\leq C, \ \ for \ \ t>0.
\end{eqnarray*}
This completes the proof of Lemma \ref{3.5}.
\end{proof}
\noindent\textit{Proof of Theorem  \ref{main-decay-bounded} (a).} Now this is immediate: passing to a subsequence, the approximate solutions $a^n$, $\theta^n$ and $u^n$ converge to a weak solution $(a, \theta, u)$ of the Boussinesq system (\ref{rewrite-boussinesq}) with $\kappa=1$ and $\nu=0$ in $L^2 _{loc}(\mathbb{R}^+, \mathbb{R}^3)$. Moreover, the estimates (\ref{c18}), (\ref{u}) and (\ref{n-th}) imply that $a^n$, $\theta^n$ and $u^n$  satisfy estimates of the form
\begin{eqnarray*}
\|v^n\|_2\leq f(t),\ \ \ for\ all\  t>0,
\end{eqnarray*}
where $f(t)$ is a continuous function independent on $n$. Then the same estimate
must hold for the limit $a^n$, $\theta^n$ and $u^n$, except possibly points in a set of measure zero.
But since weak solutions are necessarily continuous from $[0, \infty)$ to $L^2$ under the
weak topology, $\|\theta(t)\|_2$ and $\|u(t)\|_2$ are lower semi-continuous, and hence they satisfy
the above estimates for all $t > 0$.

\vskip 3mm
The combination of Lemma \ref{3.5}, Proposition \ref{1.1} and Lemma \ref{u-v-tidu} gives the conclusion of Theorem \ref{main-decay-bounded} (b). For proving the global existence in Theorem \ref{main-decay-bounded} (c)
and (d), we need the following uniform estimates for the solution
$(a, u, \theta, \nabla\Pi)$ of (\ref{rewrite-boussinesq}) with $\kappa=1$ and $\nu=0$ .
\begin{lem}\label{4.2}
Assume that $(\theta_0, u_0)$ satisfies the condition of Theorem
\ref{main-decay-bounded} (b), and additionally that $\theta_0\in
\dot{B}^{-\frac{3}{2}}_{2,1}(\mathbb{R}^3)$ and there
exist a absolute constants  $M_2>0$ and a large time $T_0^*$ fulfilling
\begin{eqnarray}\int_0^{T_0^*}\|\nabla u\|_{\dot{B}_{2,1}^{\frac{3}{2}}}d\tau<M_2\label{u-satisfy-1}.
\end{eqnarray}
Then the solution $(a, u, \theta, \nabla\Pi)$ of (\ref{rewrite-boussinesq}) with $\kappa=1$ and $\nu=0$ satisfies the following
uniform estimates
\begin{eqnarray}\label{h0}
    \|a\|_{\tilde{L}^\infty(\mathbb{R}^+;\dot{B}_{2,1}^{\frac{3}{2}})}
     \leq C,
\end{eqnarray}

\begin{eqnarray}\label{h7}
    \|\theta\|_{\tilde{L}^\infty(\mathbb{R}^+;\dot{B}_{2,1}^{-\frac{3}{2}})}
     +\|\theta\|_{L^1(\mathbb{R}^+;{B}_{2,1}^{\frac{1}{2}})}
     \leq C,
\end{eqnarray}
\begin{eqnarray}\label{h8}
    \|u\|_{\tilde{L}^\infty(\mathbb{R}^+;B_{2,1}^{\frac{1}{2}})}
     +\|u\|_{L^1(\mathbb{R}^+;\dot{B}_{2,1}^{\frac{5}{2}})}
     +\|\nabla\Pi\|_{L^1(\mathbb{R}^+;B_{2,1}^{\frac{1}{2}})}
     \leq C,
\end{eqnarray}
for some positive constant $C$.
\end{lem}
\begin{proof}
Using Remark
\ref{2.2}(2) (iii) and Lemma \ref{3.9}, we obtain
\begin{eqnarray*}
  &&\|u\|_{\dot{B}_{2,1}^{\frac{1}{2}}}
     \leq\|u\|_2^{\frac{1}{2}}\|\nabla u\|_2^{\frac{1}{2}},\\
  &&\|\theta\|_{\dot{B}_{2,1}^{\frac{1}{2}}}
      \leq \|\theta\|_2^{\frac{1}{2}}\|\nabla\theta\|_2^{\frac{1}{2}},\\
  &&\|\theta\|_2\leq C\|\theta_0\|_1(1+t)^{-3/4}\leq C\varepsilon_0(1+t)^{-3/4},\\
  &&\|\nabla\theta\|_2\leq C (1+t)^{-7/4}.
\end{eqnarray*}
From (\ref{u-bounded}) and the above inequalities, we deduce
\begin{eqnarray*}
    \int_0^t \|\theta\|_{\dot{B}_{2,1}^{\frac{1}{2}}}d\tau
     \leq C\varepsilon_0 , \ \ {\rm for} \ t>0,
\end{eqnarray*}
\begin{eqnarray*}
    \int_0^t\|u\|^4_{\dot{B}_{2,1}^{\frac{1}{2}}} d\tau
     \leq\int_0^t\|u\|_2^{2}\|\nabla u\|_2^2 d \tau \leq C  , \ \ \mathrm{for} \ t>0.
\end{eqnarray*}
Hence, for any $\epsilon> 0$, there exists $T_0(\epsilon)> 0$
such that
\begin{eqnarray}\label{d10}
   \|u\left(T_0(\epsilon)\right)\|_{\dot{B}_{2,1}^{\frac{1}{2}}}<\epsilon.
\end{eqnarray}
On the other hand, applying Lemma \ref{prop:estimate tranport equation} to the transport equation in (\ref{rewrite-boussinesq}) gives
\begin{eqnarray*}
  \|a\|_{\tilde{L}^\infty_{t}(\dot{B}^{\frac{3}{2}}_{2,1})}
     \leq \|a_0\|_{\dot{B}^{\frac{3}{2}}_{2,1}}+C_1\int_0^t \|a(\tau)\|_{\dot{B}^{\frac{3}{2}}_{2,1}}\|\nabla u(\tau)\|_{\dot{B}^{\frac{3}{2}}_{2,1}}d\tau
\end{eqnarray*}
for any $t>0$. Then applying Gronwall inequality yields
\begin{eqnarray*}
  \|a(t)\|_{\tilde{L}^\infty_{[0,T_0]}(\dot{B}^{\frac{3}{2}}_{2,1})}
     \leq \|a_0\|_{\dot{B}^{\frac{3}{2}}_{2,1}}\exp\{C_1\int_0^{T_0} \|\nabla u(\tau)\|_{\dot{B}^{\frac{3}{2}}_{2,1}}d\tau\}\leq C \eta_0,
\end{eqnarray*}
here $T_0=\min\{T_0(\epsilon), T_0^*\}> 0.$
Following the same line, it is easy to observe that for $t\geq T_0$
\begin{eqnarray}\label{a-estimate-T0}
 \|a\|_{\tilde{L}^\infty_{[T_0,t]}(\dot{B}^{\frac{3}{2}}_{2,1})}
     &\leq & \|a(T_0)\|_{\dot{B}^{\frac{3}{2}}_{2,1}}
     +C \|a(\tau)\|_{\tilde{L}^\infty_{[T_0,t]}(\dot{B}^{\frac{3}{2}}_{2,1})}\|\nabla u(\tau)\|_{\tilde{L}^1_{[T_0,t]}(\dot{B}^{\frac{3}{2}}_{2,1})}\nonumber\\
    &\lesssim & \eta_0+\|a(\tau)\|_{\tilde{L}^\infty_{[T_0,t]}(\dot{B}^{\frac{3}{2}}_{2,1})}\|\nabla u(\tau)\|_{\tilde{L}^1_{[T_0,t]}(\dot{B}^{\frac{3}{2}}_{2,1})}.
\end{eqnarray}
Note that for $a$ small, we can rewrite the momentum equation in (\ref{rewrite-boussinesq}) as
\begin{eqnarray*}
\partial_t u+ (u\cdot\nabla u)-\Delta u+ \nabla
\Pi=a(\Delta u-\nabla
\Pi)+\theta e_3
\end{eqnarray*}
Applying Lemma \ref{prop:estimate tranport equation} to the above equation, we obtain
\begin{eqnarray}\label{d11}
   && \|u\|_{\tilde{L}^\infty_{[T_0,t]}(\dot{B}^{\frac{1}{2}}_{2,1})}
      +\|u\|_{L^1_{[T_0,t]}(\dot{B}^{\frac{5}{2}}_{2,1})}
      +\|\nabla\Pi\|_{L^1_{[T_0,t]} (\dot{B}^{\frac{1}{2}}_{2,1})}\nonumber\\
   &&\lesssim\|u(T_0)\|_{\dot{B}^{\frac{1}{2}}_{2,1}}
     +\|u\|_{\tilde{L}^\infty_{[T_0,t]}(\dot{B}^{\frac{1}{2}}_{2,1})}
     \|\nabla u\|_{L^1_{[T_0,t]}(\dot{B}^{\frac{3}{2}}_{2,1})}
     +\|\theta\|_ {L^1_{[T_0,t]}(\dot{B}_{2,1}^{\frac{1}{2}})}\nonumber\\
     &&\quad+\|a\|_{\tilde{L}^\infty_{[T_0,t]}(\dot{B}^{\frac{3}{2}}_{2,1})}(\|\nabla u\|_{\tilde{L}^1_{[T_0,t]}(\dot{B}^{\frac{3}{2}}_{2,1})}+\|\nabla\Pi\|_{L^1_{[T_0,t]} (\dot{B}^{\frac{1}{2}}_{2,1})})
     \nonumber\\
   &&\lesssim\|u(T_0)\|_{\dot{B}^{\frac{1}{2}}_{2,1}}
     +\|u\|_{\tilde{L}^\infty_{[T_0,t]}(\dot{B}^{\frac{1}{2}}_{2,1})}
     \|\nabla u\|_{L^1_{[T_0,t]}(\dot{B}^{\frac{3}{2}}_{2,1})}+\varepsilon_0\nonumber\\
     &&\quad +\|a\|_{\tilde{L}^\infty_{[T_0,t]}(\dot{B}^{\frac{3}{2}}_{2,1})}(\|\nabla u\|_{\tilde{L}^1_{[T_0,t]}(\dot{B}^{\frac{3}{2}}_{2,1})}+\|\nabla\Pi\|_{L^1_{[T_0,t]} (\dot{B}^{\frac{1}{2}}_{2,1})}).
\end{eqnarray}
 For any $ t\geq T_0$, denote
\begin{eqnarray*}
    Z(t)=\|a\|_{\tilde{L}^\infty_{[T_0,t]}(\dot{B}^{\frac{3}{2}}_{2,1})}+\|u\|_{\tilde{L}^\infty_{[T_0,t]}(\dot{B}^{\frac{1}{2}}_{2,1})}
      +\|u\|_{L^1_{[T_0,t]}(\dot{B}^{\frac{5}{2}}_{2,1})}
      +\|\nabla\Pi\|_{L^1_{[T_0,t]}(\dot{B}^{\frac{1}{2}}_{2,1})},
\end{eqnarray*}
It follows from (\ref{a-estimate-T0}) and (\ref{d11}) that
\begin{eqnarray*}
    Z(t)\leq C_3\eta_0+\|u(T_0)\|_{\dot{B}^{\frac{1}{2}}_{2,1}}
     +C\varepsilon_0+C_2 Z(t)^2.
\end{eqnarray*}
Let
\begin{eqnarray}\label{d12}
    T=\sup_{t>T_0}\{t: Z(t)
     \leq 2( C_3\eta_0+\|u(T_0)\|_{\dot{B}^{\frac{1}{2}}_{2,1}}+C\varepsilon_0)\},
\end{eqnarray}
and we claim that $T=\infty$. Indeed, if $T<\infty$, taking
$\epsilon \leq \frac{1}{32C_2}$ in (\ref{d10}), $\eta_0\leq \frac{1}{32C_2C_3}$ and
$\varepsilon_0\leq\frac{1}{32C_2C}$,
 we have
\begin{eqnarray}\label{d13}
    Z(t)
    \leq \frac{3}{2}(C_3\eta_0+\|u(T_0)\|_{\dot{B}^{\frac{1}{2}}_{2,1}}+C\varepsilon_0) \ \
      \mathrm{for}\ \ T_0\leq t\leq T.
\end{eqnarray}
This contradicts (\ref{d12}), and therefore $T=\infty $. So we have
\begin{eqnarray}\label{a-t-estimate}
   \|a(t)\|_{\tilde{L}^\infty(R^+;\dot{B}^{\frac{3}{2}}_{2,1})}
     &\leq & \|a_0\|_{\dot{B}^{\frac{3}{2}}_{2,1}}\exp\{C_1\int_0^{T_0} \|\nabla u(\tau)\|_{\dot{B}^{\frac{3}{2}}_{2,1}}d\tau\}\nonumber\\
     &&+2(C_3\eta_0+\|u(T_0)\|_{\dot{B}^{\frac{1}{2}}_{2,1}}+C\varepsilon_0)=c_1
\end{eqnarray}
and
\begin{eqnarray*}
   \|u\|_{\tilde{L}^\infty (\mathbb{R}^+;\dot{B}^{\frac{1}{2}}_{2,1})}
    +\|u\|_{L^1(\mathbb{R}^+;\dot{B}^{\frac{5}{2}}_{2,1})}
    +\|\nabla\Pi\|_{L^1(\mathbb{R}^+;\dot{B}^{\frac{1}{2}}_{2,1})}
   \leq C,
\end{eqnarray*}
together with (\ref{pressure}) and $\|a\|_2\leq C,\ \|u\|_2\leq C$, we can easily get (\ref{h0}) and (\ref{h8}) which are the part proof of Theorem
\ref{main-decay-bounded} (c).

Note that for $a$ small, we can rewrite the temperature equation in (\ref{rewrite-boussinesq}) as
\begin{eqnarray*}
\partial_t \theta+ (u\cdot\nabla \theta)-\Delta
\theta=\nabla(a\nabla\theta)
\end{eqnarray*}
Moreover, from Lemma \ref{prop:estimate tranport equation} and Lemma \ref{a-theta-product-estimate}, it follows that
\begin{eqnarray*}
   \|\theta\|_{\tilde{L}_t^\infty(\dot{B}_{2,1}^{-\frac{3}{2}})}
    +\|\theta\|_{L^1_t(\dot{B}_{2,1}^{\frac{1}{2}})}
   &\lesssim &\|\theta_0\|_{\dot{B}_{2,1}^{-\frac{3}{2}}}+
    \int_0^t\|\nabla u\|_{\dot{B}^{\frac{3}{2}}_{2,1}}\|\theta(t')\|_{
    \tilde{L}_{t'}^\infty(\dot{B}_{2,1}^{-\frac{3}{2}})}dt'
    + \|\nabla(a\nabla\theta)\|_{L^1_t(\dot{B}_{2,1}^{-\frac{3}{2}})}\\
     &\lesssim &\|\theta_0\|_{\dot{B}_{2,1}^{-\frac{3}{2}}}+
    \int_0^t\|\nabla u\|_{\dot{B}^{\frac{3}{2}}_{2,1}}\|\theta(t')\|_{
    \tilde{L}_{t'}^\infty(\dot{B}_{2,1}^{-\frac{3}{2}})}dt'\\
    &&+\|\theta\|_{L^1_t(\dot{B}_{2,1}^{\frac{1}{2}})}\|a\|_{\tilde{L}^\infty(R^+;\dot{B}^{\frac{3}{2}}_{2,1})}.
\end{eqnarray*}
Using (\ref{h8}),(\ref{a-t-estimate}) and the Gronwall inequality, there exists a constant $c_0$ satisfying,
\begin{eqnarray*}
    \|\theta\|_{\tilde{L}^\infty(\mathbb{R}^+;\dot{B}_{2,1}^{-\frac{3}{2}})}
      +c_0\|\theta(t)\|_{L^1(\mathbb{R}^+;\dot{B}_{2,1}^{\frac{1}{2}})}
    \lesssim \|\theta_0\|_{\dot{B}_{2,1}^{-\frac{3}{2}}}
     \exp \{C\int_0^{\infty}\|\nabla u(\tau)\|_{\dot{B}_{2,1}^{\frac{3}{2}}}d\tau \}
    \leq C.
\end{eqnarray*}
Combining the above inequality and $\|\theta\|_2\leq C(1+t)^{\frac{5}{4}}$, we can easily get (\ref{h7}).

%
\end{proof}

\begin{lem}\label{4.4}
Assume that $(a_0, \theta_0, u_0)$ satisfies the conditions of Theorem
\ref{main-decay-bounded} (d). Then the solution $(a,\theta,u,\Pi)$ of the
 Boussinesq system (\ref{rewrite-boussinesq}) with $\kappa=1$ and $\nu=0$  satisfies the following
uniform estimates
\begin{eqnarray*}
    \|\theta\|_{\tilde{L}^\infty(\mathbb{R}^+;\dot{B}_{2,1}^{-\frac{1}{2}})}
     +\|\theta\|_{L^1(\mathbb{R}^+;{B}_{2,1}^{\frac{3}{2}})}
     \leq C,
\end{eqnarray*}
\begin{eqnarray*}
   \|a\|_{\tilde{L}^\infty(\mathbb{R}^+;{B}_{2,1}^{\frac{5}{2}})}+ \|u\|_{\tilde{L}^\infty(\mathbb{R}^+;B_{2,1}^{\frac{3}{2}})}
     +\|u\|_{L^1(\mathbb{R}^+;\dot{B}_{2,1}^{\frac{7}{2}})}
     +\|\nabla\Pi\|_{L^1(\mathbb{R}^+;B_{2,1}^{\frac{3}{2}})}
     \leq C,
\end{eqnarray*}
for some constant $C> 0$ independent of $t$.
\end{lem}
\begin{proof}
Applying Lemma \ref{prop:estimate tranport equation} and Lemma \ref{3.1} to the density equation and following the same line as in the proof of Lemma \ref{4.2}, there holds
\begin{eqnarray}\label{smooth-a-estimate}
   \|a\|_{\tilde{L}^\infty(\mathbb{R}^+;\dot{B}_{2,1}^{\frac{5}{2}})}
   \lesssim
   \|a_0\|_{\dot{B}_{2,1}^{\frac{5}{2}}}
    \exp\{C\int_0^{\infty}\|\nabla u(\tau)\|_{\dot{B}_{2,1}^{\frac{3}{2}}}d\tau \}
   \leq C.
\end{eqnarray}
Applying Lemma \ref{prop:estimate tranport equation} and Lemma \ref{3.1} to the temperature equation, the velocity equation and following the same line as in the proof of Proposition \ref{4.2}, there holds
\begin{eqnarray*}
   &&\|\theta\|_{\tilde{L}^\infty(\mathbb{R}^+;\dot{B}_{2,1}^{-\frac{1}{2}})}
    +\|\theta\|_{L^1(\mathbb{R}^+;\dot{B}_{2,1}^{\frac{3}{2}})}\\
    &&\lesssim \|\theta_0\|_{\dot{B}_{2,1}^{-\frac{1}{2}}}+\int_0^{t}\|\theta\|_{\tilde{L}^\infty(\mathbb{R}^+;\dot{B}_{2,1}^{-\frac{1}{2}})}\|\nabla u(\tau)\|_{\dot{B}_{2,1}^{\frac{3}{2}}}d\tau
    +\|\nabla(a\nabla\theta)\|_{L^1(\dot{B}_{2,1}^{-\frac{1}{2}})}\\
    &&\lesssim\|\theta_0\|_{\dot{B}_{2,1}^{-\frac{1}{2}}}+\int_0^{t}\|\theta\|_{\tilde{L}^\infty(\mathbb{R}^+;\dot{B}_{2,1}^{-\frac{1}{2}})}\|\nabla u(\tau)\|_{\dot{B}_{2,1}^{\frac{3}{2}}}d\tau\\
    &&\quad+\|a\|_{\tilde{L}^\infty(\mathbb{R}^+;\dot{B}_{2,1}^{\frac{3}{2}})}\|\theta\|_{L^1(\mathbb{R}^+;\dot{B}_{2,1}^{\frac{3}{2}})}
 \end{eqnarray*}
  and
\begin{eqnarray*}
    &&\|u\|_{\tilde{L}^\infty(\mathbb{R}^+;\dot{B}_{2,1}^{\frac{3}{2}})}
    +\|u\|_{L^1(\mathbb{R}^+;\dot{B}_{2,1}^{\frac{7}{2}})} +\|\nabla\Pi\|_{L^1(\mathbb{R}^+;\dot{B}_{2,1}^{\frac{3}{2}})}\\
   &&\lesssim\|u_0\|_{\dot{B}_{2,1}^{\frac{3}{2}}}+\int_0^{t}\|u\|_{\tilde{L}^\infty(\mathbb{R}^+;\dot{B}_{2,1}^{\frac{3}{2}})}\|\nabla u(\tau)\|_{\dot{B}_{2,1}^{\frac{3}{2}}}d\tau+\|\theta\|_{L^1(\mathbb{R}^+;\dot{B}_{2,1}^{\frac{3}{2}})}\\
    &&\quad+\|a\nabla\Pi\|_{L^1(\dot{B}_{2,1}^{\frac{3}{2}})}+\| a\Delta u\|_{L^1(\dot{B}_{2,1}^{\frac{3}{2}})}
   \\
   &&\lesssim\|u_0\|_{\dot{B}_{2,1}^{\frac{3}{2}}}+\int_0^{t}\|u\|_{\tilde{L}^\infty(\mathbb{R}^+;\dot{B}_{2,1}^{\frac{3}{2}})}\|\nabla u(\tau)\|_{\dot{B}_{2,1}^{\frac{3}{2}}}d\tau+\|\theta\|_{L^1(\mathbb{R}^+;\dot{B}_{2,1}^{\frac{3}{2}})}
   \\
   &&\quad+\|a\|_{\tilde{L}^\infty(\mathbb{R}^+;\dot{B}_{2,1}^{\frac{3}{2}})}\|\nabla\Pi\|_{L^1(\mathbb{R}^+;\dot{B}_{2,1}^{\frac{3}{2}})}
   +\|a\|_{\tilde{L}^\infty(\mathbb{R}^+;\dot{B}_{2,1}^{\frac{3}{2}})}\|u\|_{L^1(\mathbb{R}^+;\dot{B}_{2,1}^{\frac{7}{2}})}.
\end{eqnarray*}
Using (\ref{h8}),(\ref{a-t-estimate}) and the Gronwall inequality, we get
 \begin{eqnarray}\label{smooth-theta-estimate}
  &&\|\theta\|_{\tilde{L}^\infty(\mathbb{R}^+;\dot{B}_{2,1}^{-\frac{1}{2}})}
    +\|\theta\|_{L^1(\mathbb{R}^+;\dot{B}_{2,1}^{\frac{3}{2}})}\nonumber\\
    &&\lesssim
   \|\theta_0\|_{\dot{B}_{2,1}^{-\frac{1}{2}}}
    \exp\{C\int_0^{\infty}\|\nabla u(\tau)\|_{\dot{B}_{2,1}^{\frac{3}{2}}}d\tau \}
   \leq C,
\end{eqnarray}
and
\begin{eqnarray}\label{smooth-u-estimate}
  &&\|u\|_{\tilde{L}^\infty(\mathbb{R}^+;\dot{B}_{2,1}^{\frac{3}{2}})}
    +\|u\|_{L^1(\mathbb{R}^+;\dot{B}_{2,1}^{\frac{7}{2}})} +\|\nabla\Pi\|_{L^1(\mathbb{R}^+;\dot{B}_{2,1}^{\frac{3}{2}})}\nonumber\\
  &&\lesssim \|u_0\|_{\dot{B}_{2,1}^{\frac{3}{2}}}\exp \{C\int_0^{\infty}\|\nabla u(\tau)\|_{\dot{B}_{2,1}^{\frac{3}{2}}}d\tau \}+\|\theta(t)\|_{L^1(\dot{B}_{2,1}^{\frac{3}{2}})}
    \leq C.
\end{eqnarray}
These along with Remark \ref{2.2} complete the proof of the lemma.
\end{proof}

With lemmata \ref{3.5}-\ref{4.4}, it is easy to prove by a classical argument that
\begin{eqnarray*}
a\in C_b([0,\infty),B_{2,1}^{\frac{5}{2}}),\quad \theta\in C_b([0,\infty),\dot{B}_{2,1}^{-\frac{1}{2}}),\quad u\in C_b([0,\infty),B_{2,1}^{\frac{3}{2}}).
\end{eqnarray*}
This completes the proof of Theorem \ref{main-decay-bounded}.

\section{Proof of Theorem \ref{1.2}}
\setcounter{equation}{0}

\textit{Proof of Theorem \ref{1.2} (a)}. For simplicity, we just provide
some necessary a priori estimates here. Indeed let $\tilde{a}=
a-\bar{a}, \tilde{u}=
u-\bar{u}, \ \tilde{\theta}=\theta-\bar{\theta}$, then
$(a,\tilde{\theta},\tilde{u},\tilde{\Pi})$ solves
\begin{eqnarray}\label{d15}
  \left\{
  \begin{array}{lll}
      \partial_t a+(\bar{u}+\tilde{u})\cdot\nabla a=0,\\
      \partial_t \tilde{\theta}-\Delta \tilde{\theta}
       =-(\bar{u}+\tilde{u})\cdot\nabla \tilde{\theta}-\tilde{u}\cdot\nabla \bar{\theta}+\nabla(a\nabla\tilde{\theta})+\nabla((a-\bar{a})\nabla\bar{\theta}),\\
     \partial_t \tilde{u}- \Delta \tilde{u}+\nabla \tilde{\Pi}
       =-(\bar{u}+\tilde{u})\cdot\nabla\tilde{u}-\tilde{u}\cdot\nabla \bar{u}+a(\Delta\tilde{u}-\nabla \tilde{\Pi})+(a-\bar{a})(\Delta\bar{u}-\nabla\bar{\Pi})+\tilde{\theta}e_3,\\
     \mathrm{div} \tilde{u}=0,\\
     (\tilde{\theta},\tilde{u})|_{t=0}=(\tilde{\theta}_0, \tilde{u}_0).
  \end{array}
  \right.
\end{eqnarray}
Applying Lemma \ref{prop:estimate tranport equation} to the density equation in (\ref{d15})
and Lemma \ref{3.1}, we obtain
\begin{eqnarray*}
\|a\|_{\tilde{L}_t^\infty(\dot{B}_{2,1}^{3/2})}& \leq & \|a_0\|_{\dot{B}_{2,1}^{3/2}}+C\int_0^t \|a(\tau)\|_{\dot{B}_{2,1}^{3/2}}\|\bar{u}(\tau)\|_{\dot{B}_{2,1}^{5/2}}d\tau\\
  &&+C \|a\|_{\tilde{L}_t^\infty(\dot{B}_{2,1}^{3/2})}\|\tilde{u}\|_{L_t^1(\dot{B}_{2,1}^{5/2})},
\end{eqnarray*}
which along with (\ref{h8}) and the Gronwall inequality ensures
\begin{eqnarray*}
\|a\|_{\tilde{L}_t^\infty(\dot{B}_{2,1}^{3/2})} \leq  C \left(\|a_0\|_{\dot{B}_{2,1}^{3/2}}+ \|a\|_{\tilde{L}_t^\infty(\dot{B}_{2,1}^{3/2})}\|\tilde{u}\|_{L_t^1(\dot{B}_{2,1}^{5/2})}\right).
\end{eqnarray*}
On the other hand, applying Lemma \ref{3.1} and Lemma \ref{prop:estimate tranport equation} to the velocity equation in (\ref{d15}) , we get that
\begin{eqnarray}\label{d17}
     &&\|\tilde{u}\|_{\tilde{L}_t^\infty(\dot{B}^{\frac{1}{2}}_{2,1})}
       +\|\tilde{u}\|_{L_t^1(\dot{B}^{\frac{5}{2}}_{2,1})}
       +\|\nabla\tilde{\Pi}\|_{L^1_t(\dot{B}^{\frac{1}{2}}_{2,1})}\nonumber\\
     &&\lesssim
       \|\tilde{u}_0\|_{\dot{B}^{\frac{1}{2}}_{2,1}}
       +\|\tilde{u}\|_{\tilde{L}_t^\infty(\dot{B}^{\frac{1}{2}}_{2,1})}\|\nabla\tilde{u}\|_{L^1_t(\dot{B}^{\frac{3}{2}}_{2,1})}
       +\int_0^t\|\tilde{u}\|_{\tilde{L}_s^\infty(\dot{B}^{\frac{1}{2}}_{2,1})}\|\nabla\bar{u}(s)\|_{\dot{B}^{\frac{3}{2}}_{2,1}}ds +\|\tilde{\theta}\|_{L^1_t(\dot{B}^{\frac{1}{2}}_{2,1})}\nonumber\\
     &&\quad+\int_0^t\|(a-\bar{a})(\Delta\bar{u}-\nabla\bar{\Pi})\|_{\dot{B}_{2,1}^{1/2}}d\tau +\int_0^t\|a(\Delta\tilde{u}-\nabla \tilde{\Pi})\|_{\dot{B}_{2,1}^{1/2}}d\tau  .
\end{eqnarray}
Using Gronwall inequality and (\ref{h8}), we obtain
\begin{eqnarray}\label{d18}
   &&\|\tilde{u}\|_{\tilde{L}_t^\infty(\dot{B}^{\frac{1}{2}}_{2,1})}
     +\|\tilde{u}\|_{L^1_t(\dot{B}^{\frac{5}{2}}_{2,1})}
     +\|\nabla\tilde{\Pi}\|_{\tilde{L}_t^1(\dot{B}^{\frac{1}{2}}_{2,1})}\nonumber\\
   &&\lesssim
     \|\tilde{u}_0\|_{\dot{B}^{\frac{1}{2}}_{2,1}}+\|\tilde{u}\|_{\tilde{L}_t^\infty
     (\dot{B}^{\frac{1}{2}}_{2,1})}\|\tilde{u}\|_{L^1_t(\dot{B}^{\frac{5}{2}}_{2,1})}
     +\|\tilde{\theta}\|_{L^1_t(\dot{B}^{\frac{1}{2}}_{2,1})}\nonumber\\
   &&\quad+\|a\|_{\tilde{L}_t^\infty(\dot{B}_{2,1}^{3/2})}\left(\|\tilde{u}\|_{L^1_t(\dot{B}^{\frac{5}{2}}_{2,1})}+\|\nabla\tilde{\Pi}\|_{\tilde{L}_t^1(\dot{B}^{\frac{1}{2}}_{2,1})}\right)\nonumber\\
   &&\quad+\|\bar{a}\|_{\tilde{L}_t^\infty(\dot{B}_{2,1}^{3/2})}+\int_0^t\|a\|_{\dot{B}_{2,1}^{3/2}}\left(\|\bar{u}\|_{\dot{B}_{2,1}^{5/2}}+\|\nabla\bar{\Pi}\|_{\dot{B}_{2,1}^{1/2}}\right)d\tau.
\end{eqnarray}
Applying Lemma \ref{prop:estimate tranport equation} to the temperature equation in (\ref{d15}) and Lemma \ref{3.1}, we obtain
\begin{eqnarray}\label{d19}
    &&\|\tilde{\theta}\|_{\tilde{L}_t^\infty(\dot{B}_{2,1}^{-\frac{3}{2}})}
      +\|\tilde{\theta}\|_{L^1_t(\dot{B}_{2,1}^{\frac{1}{2}})}\nonumber\\
    &&\lesssim
      \|\tilde{\theta}_0\|_{\dot{B}^{-\frac{3}{2}}_{2,1}}
      +\int_0^t\|\tilde{\theta}\|_{\tilde{L}_s^\infty(\dot{B}_{2,1}^{-\frac{3}{2}})}\|\nabla\bar{u}(s)\|_{\dot{B}^{\frac{3}{2}}_{2,1}}ds\nonumber\\
    && \quad+\|\tilde{\theta}\|_{\tilde{L}_t^\infty(\dot{B}_{2,1}^{-\frac{3}{2}})}
      \|\nabla\tilde{u}\|_{L^1_t(\dot{B}^{\frac{3}{2}}_{2,1})}+\|\tilde{u}\|_{\tilde{L}_t^\infty(\dot{B}^{\frac{1}{2}}_{2,1})}
      \|\bar{\theta}\|_{L^1_t(\dot{B}_{2,1}^{\frac{1}{2}})}
      \nonumber\\
    &&\quad+\int_0^t\|\nabla(a\nabla\tilde{\theta})\|_{\dot{B}_{2,1}^{-\frac{3}{2}}}d\tau+\int_0^t\|\nabla((a-\bar{a})\nabla\bar{\theta})\|_{\dot{B}_{2,1}^{-\frac{3}{2}}}d\tau
     \nonumber\\
    &&\lesssim
      \|\tilde{\theta}_0\|_{\dot{B}^{-\frac{3}{2}}_{2,1}}
      +\int_0^t\|\tilde{\theta}\|_{\tilde{L}_s^\infty(\dot{B}_{2,1}^{-\frac{3}{2}})}\|\nabla\bar{u}(s)\|_{\dot{B}^{\frac{3}{2}}_{2,1}}ds
      \nonumber\\
     && \quad +\|\tilde{\theta}\|_{\tilde{L}_t^\infty(\dot{B}_{2,1}^{-\frac{3}{2}})}
      \|\nabla\tilde{u}\|_{L^1_t(\dot{B}^{\frac{3}{2}}_{2,1})}+\|\tilde{u}\|_{\tilde{L}_t^\infty(\dot{B}^{\frac{1}{2}}_{2,1})}
      \|\bar{\theta}\|_{L^1_t(\dot{B}_{2,1}^{\frac{1}{2}})}\nonumber\\
    &&\quad+\|a\|_{\tilde{L}_t^\infty(\dot{B}_{2,1}^{3/2})}\|\tilde{\theta}\|_{L^1_t(\dot{B}_{2,1}^{\frac{1}{2}})}+\|\bar{a}\|_{\tilde{L}_t^\infty(\dot{B}_{2,1}^{3/2})}+\int_0^t\|a\|_{\tilde{L}_t^\infty(\dot{B}_{2,1}^{3/2})}\|\bar{\theta}\|_{\dot{B}_{2,1}^{\frac{1}{2}}}ds.
  \end{eqnarray}
Then the Gronwall inequality and (\ref{h8}) yield
\begin{eqnarray}\label{d20}
   &&\|\tilde{\theta}\|_{\tilde{L}_t^\infty(\dot{B}_{2,1}^{-\frac{3}{2}})}
     +\|\tilde{\theta}\|_{\tilde{L}_t^1(\dot{B}_{2,1}^{\frac{1}{2}})}\nonumber\\
   &&\lesssim
     \|\tilde{\theta}_0\|_{\dot{B}^{-\frac{3}{2}}_{2,1}}
     +\|\tilde{\theta}\|_{\tilde{L}_t^\infty(\dot{B}_{2,1}^{-\frac{3}{2}})}
     \|\tilde{u}\|_{L^1_t(\dot{B}^{\frac{5}{2}}_{2,1})}
     +\|\tilde{u}\|_{\tilde{L}_t^\infty(\dot{B}^{\frac{1}{2}}_{2,1})}
     \|\bar{\theta}\|_{\tilde{L}_t^1(\dot{B}_{2,1}^{\frac{1}{2}})}\nonumber\\
    &&\quad+\|a\|_{\tilde{L}_t^\infty(\dot{B}_{2,1}^{3/2})}\|\tilde{\theta}\|_{L^1_t(\dot{B}_{2,1}^{\frac{1}{2}})}
    +\|\bar{a}\|_{\tilde{L}_t^\infty(\dot{B}_{2,1}^{3/2})}+\int_0^t\|a\|_{\tilde{L}_t^\infty(\dot{B}_{2,1}^{3/2})}\|\bar{\theta}\|_{\dot{B}_{2,1}^{\frac{1}{2}}}ds.
\end{eqnarray}
 Let
\begin{eqnarray*}
    \tilde{Z}(t)
    =\|a\|_{\tilde{L}^\infty_t(\dot{B}^{\frac{3}{2}}_{2,1})}
    +\|\tilde{u}\|_{\tilde{L}^\infty_t(\dot{B}^{\frac{1}{2}}_{2,1})}
     +\|\tilde{u}\|_{\tilde{L}^1_t(\dot{B}^{\frac{5}{2}}_{2,1})}
     +\|\nabla\tilde{\Pi}\|_{\tilde{L}^1_t(\dot{B}^{\frac{1}{2}}_{2,1})}
     +\|\tilde{\theta}\|_{\tilde{L}^\infty_t(\dot{B}_{2,1}^{-\frac{3}{2}})}
     +\|\tilde{\theta}\|_{\tilde{L}^1_t(\dot{B}^{\frac{1}{2}}_{2,1})}.
\end{eqnarray*}
Combining (\ref{d18}) with (\ref{d20}), we obtain
\begin{eqnarray*}
    \tilde{Z}(t) \leq C_3(\|\tilde{a}_0\|_{\dot{B}^{\frac{3}{2}}_{2,1}}+\|\tilde{u}_0\|_{\dot{B}^{\frac{1}{2}}_{2,1}}
    +\|\tilde{\theta}_0\|_{\dot{B}^{-\frac{3}{2}}_{2,1}}+\|\bar{a}\|_{\tilde{L}_t^\infty(\dot{B}_{2,1}^{3/2})}
    +\tilde{Z}(t)^2),
\end{eqnarray*}
for some $C_3>0$. Then a similar derivation of (\ref{d13}) shows that if
\begin{eqnarray*}
    \|\tilde{a}_0\|_{\dot{B}^{\frac{3}{2}}_{2,1}}+\|\tilde{u}_0\|_{\dot{B}^{\frac{1}{2}}_{2,1}}+\|\tilde{\theta}_0\|_{\dot{B}^{-\frac{3}{2}}_{2,1}}+c_1
    \leq \frac{1}{4C_3^2},
\end{eqnarray*}
where $c_1$ gives in (\ref{a-t-estimate}), there holds
\begin{eqnarray}\label{d21}
    \tilde{Z}(t)
    \leq
    2C_3(\|\tilde{u}_0\|_{\dot{B}^{\frac{1}{2}}_{2,1}}
    +\|\tilde{\theta}_0\|_{\dot{B}^{-\frac{3}{2}}_{2,1}}), \quad t>0.
\end{eqnarray}
So we have
\begin{eqnarray}\label{q1}
   &&\|a\|_{\tilde{L}^\infty_t(\dot{B}^{\frac{3}{2}}_{2,1})}+\|\tilde{u}\|_{\tilde{L}^\infty(\mathbb{R}^+;\dot{B}^{\frac{1}{2}}_{2,1})}
   +\|\tilde{u}\|_{L^1(\mathbb{R}^+;\dot{B}^{\frac{5}{2}}_{2,1})}
+\|\nabla\tilde{\Pi}\|_{L^1(\mathbb{R}^+;\dot{B}_{2,1}^{\frac{1}{2}})}
   \nonumber\\
   &&+\|\tilde{\theta}\|_{\tilde{L}^\infty(\mathbb{R}^+;\dot{B}_{2,1}^{-\frac{3}{2}})}
   +\|\tilde{\theta}\|_{L^1(\mathbb{R}^+;\dot{B}^{\frac{1}{2}}_{2,1})}
   \lesssim \varepsilon_1.
\end{eqnarray}
With (\ref{q1}), we can prove the propagation of regularity for smoother initial data. Similar as the proof in (\ref{smooth-a-estimate}), (\ref{smooth-theta-estimate}) and (\ref{smooth-u-estimate}), combining (\ref{h7})-(\ref{h8}), (\ref{a-t-estimate}) and (\ref{q1}), we easily get the
\begin{eqnarray*}
a\in \tilde{L}^\infty(\dot{B}^{\frac{5}{2}}_{2,1}),\quad \theta\in \tilde{L}^\infty(\mathbb{R}^+;\dot{B}_{2,1}^{-\frac{1}{2}}),\quad u\in \tilde{L}^\infty(\mathbb{R}^+;\dot{B}^{\frac{3}{2}}_{2,1})
\end{eqnarray*}
because $\tilde{a}=
a-\bar{a}, \tilde{u}=
u-\bar{u}, \ \tilde{\theta}=\theta-\bar{\theta}$. Noticing that$(a, \theta, u)\in (L^2)^3$, this along with the above result shows that
$$
a\in \tilde{L}^\infty(\mathbb{R}^+;B_{2,1}^{\frac{5}{2}}),\quad\theta\in \tilde{L}^\infty(\mathbb{R}^+;B_{2,1}^{-\frac{1}{2}}),\quad u\in \tilde{L}^\infty(\mathbb{R}^+;B^{\frac{3}{2}}_{2,1}).
$$
Then a standard interpolation between the results Lemma \ref{4.4} and (\ref{q1}) implies (\ref{more}). This completes the proof
of Theorem \ref{1.2} (a).

We can follow the standard energy method \cite{L-L} to get the proof of Theorem \ref{1.2} (b).
$\hfill\Box$\\


\end{document}